\newif{\ifdraft}\drafttrue
\draftfalse

\ifdraft
\documentclass[a4paper,USenglish]{lipics}
\usepackage{showkeys}
\else
\documentclass[a4paper,USenglish]{lipics}
\fi 

\usepackage[shortlabels]{enumitem}
\setlist[enumerate]{nosep}
\setlist[itemize]{nosep}
\setlist[description]{nosep}
\usepackage{mathtools}
\usepackage[utf8]{inputenc}
\usepackage{amssymb,amsmath,dsfont}
\usepackage{fancyhdr}
\usepackage{stmaryrd}
\usepackage{xspace}
\usepackage{tikz}
\usepackage{multicol}
 \usepackage{tabularx}  
 \usepackage{ragged2e}  
 
\usepackage{microtype}

\newtheorem{lemma*}{Lemma}
\newtheorem*{claim*}{Claim}
\theoremstyle{theorem}
\newtheorem{proposition}[theorem]{Proposition}

\theoremstyle{remark}
\newtheorem{question}[theorem]{Question}

\newcommand{\qmp}{quotient presentation\xspace}
\newcommand{\qmps}{quotient presentations\xspace}

\newcommand{\coords}[1]{\ensuremath{\mathrm{Coord}({#1})}} 
\newcommand{\coordsj}[2]{\ensuremath{\mathrm{Coord}_{{#2}}({#1})}} 
\newcommand{\pivot}[1]{\ensuremath{\pi_{{#1}}}} 

\usepackage{amsfonts}
\usepackage{amssymb}
\usepackage{amsmath}

\usepackage{tabularx,hhline,rotating,xspace}
\usepackage{tikz}
\usetikzlibrary{calc}

\usepackage{prettyref}

\renewcommand{\setminus}{\mysetminus}

\makeindex

 \newcommand{\prref}\prettyref
\newrefformat{thm}{Theorem~\ref{#1}}
\newrefformat{lem}{Lem\-ma~\ref{#1}}
\newrefformat{def}{Definition~\ref{#1}}
\newrefformat{cor}{Corollary~\ref{#1}} 
\newrefformat{prop}{Proposition~\ref{#1}}
\newrefformat{sec}{Section~\ref{#1}}
\newrefformat{cap}{Chapter~\ref{#1}}
\newrefformat{bsp}{Example~\ref{#1}}
\newrefformat{rem}{Remark~\ref{#1}}
\newrefformat{fig}{Figure~\ref{#1}}
\newrefformat{eq}{(\ref{#1})}
\newrefformat{ex}{Example~\ref{#1}}
\newrefformat{tab}{Table~\ref{#1}}
\newrefformat{app}{Appendix~\ref{#1}}
\newrefformat{alg}{Algorithm~\ref{#1}}
\newrefformat{qu}{Question~\ref{#1}}

\newcommand{\mysetminusD}{\raisebox{.8pt}{\hbox{\tikz{\draw[line width=0.6pt,line cap=round] (3.5pt,0pt) -- (0,5.2pt);}}}}
\newcommand{\mysetminusT}{\mysetminusD}
\newcommand{\mysetminusS}{\raisebox{.5pt}{\hbox{\tikz{\draw[line width=0.45pt,line cap=round] (2.2pt,0) -- (0,3.8pt);}}}}
\newcommand{\mysetminusSS}{\raisebox{.35pt}{\hbox{\tikz{\draw[line width=0.4pt,line cap=round] (1.5pt,0) -- (0,2.8pt);}}}}

\newcommand{\mysetminus}{\mathbin{\mathchoice{\mysetminusD}{\mysetminusT}{\mysetminusS}{\mysetminusSS}}}

\newcommand{\ExtGCD}{\ensuremath{\textsc{ExtGCD}}\xspace}

\newcommand{\set}[2]{\left\{#1\; \middle|\; #2\right\}}

\newcommand{\oneset}[1]{\left\{\mathinner{#1}\right\}}

\newcommand{\abs}[1]{\left|\mathinner{#1}\right|}

\newcommand{\floor}[1]{\left\lfloor\mathinner{#1} \right\rfloor}

\newcommand{\gen}[1]{\left< \mathinner{#1} \right>}
\newcommand{\genr}[2]{\left< \, \mathinner{#1}\; \middle|\;\mathinner{#2} \, \right>}

\newcommand{\N}{\ensuremath{\mathbb{N}}}
\newcommand{\Z}{\ensuremath{\mathbb{Z}}}

 \newcommand{\NP}{\ensuremath{\mathsf{NP}}\xspace} %
 \renewcommand{\L}{\ensuremath{\mathsf{LOGSPACE}}\xspace} %
 \newcommand{\TC}{\ensuremath{\mathsf{TC}^0}\xspace}
 \newcommand{\Tc}[1]{\ensuremath{\mathsf{TC}^{#1}}\xspace}
 \newcommand{\Ac}[1]{\ensuremath{\mathsf{AC}^{#1}}\xspace}

 \newcommand{\NC}{\ensuremath{\mathsf{NC}}\xspace}
 
 \renewcommand{\P}{\ensuremath{\mathsf{P}}\xspace}

\renewcommand{\phi}{\varphi}

\newcommand{\Sig}{\Sigma}

\newcommand{\Del}{\Delta}

\newcommand{\Oh}{\mathcal{O}}

\newcommand{\cT}{\mathcal{T}}

\newcommand{\cP}{\mathcal{P}}
\newcommand{\cN}{\mathcal{N}}

\newcommand{\ssnq}{\subsetneqq}

\newcommand\Ncr{\ensuremath{\mathcal{N}_{c,r}}\xspace}

\newcommand{\sse}{\subseteq}

\newcommand{\proc}[1]{\textsc{#1}}

\newcommand\ie{i.\,e., }
\newcommand\Wlog{W.\,l.\,o.\,g.\ }
\newcommand\wwlog{w.\,l.\,o.\,g.\ }
\newcommand\eg{e.\,g.\xspace}

 \renewcommand{\theenumi}{\textup{(\roman{enumi})}}
 
 \renewcommand{\theenumiii}{\textup{\arabic{enumiii}.}}

\title{\TC circuits for algorithmic problems in nilpotent groups
}

\author[1]{Alexei Myasnikov}
\author[2]{Armin Wei\ss}
\affil[1]{Stevens Institute of Technology, Hoboken, NJ, USA}
\affil[2]{Universit\"at Stuttgart, Germany
}
\authorrunning{A. Myasnikov, A. Wei\ss} 

\Copyright{Alexei Myasnikov, Armin Wei\ss}


\begin{document}
\maketitle

\vspace{-6mm}
\begin{abstract}
	Recently, Macdonald et.\,al.\ showed that many algorithmic problems for finitely generated nilpotent groups including computation of normal forms, the subgroup membership problem, the conjugacy problem, and computation of subgroup presentations can be done in \L. Here we follow their approach and show that all these problems are complete for the uniform circuit class \TC~-- uniformly for all $r$-generated nilpotent groups of class at most $c$ for fixed $r$ and $c$.
	
	In order to solve these problems in \TC, we show that the unary version of the extended gcd problem (compute greatest common divisors and express them as linear combinations) is in \TC.
	
	Moreover, if we allow a certain binary representation of the inputs, then the word problem and computation of normal forms is still in uniform \TC, while all the other problems we examine are shown to be \TC-Turing reducible to the binary extended gcd problem.
	
	\vspace{-1mm}
\end{abstract}

\vspace{-1mm}
\tableofcontents
\renewcommand\rightmark{\TC circuits for algorithmic problems in nilpotent groups}
\renewcommand\leftmark{A. Myasnikov, A. Wei\ss}
\section{Introduction}

The word problem (given a word over the generators, does it represent the identity?) is one of the fundamental algorithmic problems in group theory introduced by Dehn in 1911 \cite{Dehn11}. While for general finitely presented groups all these problems are undecidable  \cite{nov55, boone59}, for many particular classes of groups decidability results have been established~-- not just for the word problem but also for a wide range of other problems. Finitely generated nilpotent groups are a class where many algorithmic problems are (efficiently) decidable (with some exceptions like the problem of solving equations~-- see \eg\ \cite{GarretaMO16}).

In 1958, Mal'cev~\cite{Malcev} established decidability of the word and subgroup membership problem by investigating finite approximations of nilpotent groups. In 1965, Blackburn~\cite{Blackburn} showed decidability of the conjugacy problem. However, these methods did not allow any efficient (\eg\ polynomial time) algorithms.
Nevertheless, in 1966 Mostowski provided ``practical'' algorithms for the word problem and several other problems \cite{Mostowski}.
In terms of complexity, a major step was the result by Lipton and Zalcstein \cite{Lipton_Zalc} that the word problem of linear groups is in \L. Together with the fact that finitely generated nilpotent groups are linear (see \eg\ \cite{Hal69,Kargapolov-Merzlyakov}) this gives a \L solution to the word problem of nilpotent groups, which was later improved to uniform \TC by Robinson \cite{Robinson93phd}.

A typical algorithmic approach to nilpotent groups is using so-called Mal'cev (or Hall--Mal'cev) bases (see \eg\ \cite{Hal69,Kargapolov-Merzlyakov}), which allow to carry out group operations by evaluating polynomials (see Lemma~\ref{lem:malcev}). This approach was systematically used in~\cite{KRRRCh} and \cite{Mostowski} or~-- in the more general setting of polycyclic presentations~-- in~\cite{Sim94} for solving (among others) the subgroup membership and conjugacy problem of polycyclic groups.
Recently in \cite{MNU2,MyasnikovNU16} polynomial time bounds for the equalizer and subgroup membership problems in nilpotent groups have been given. 
Finally, in \cite{MacDonaldMNV15} the following problems were shown to be in \L using the Mal'cev basis approach.
Here, $\Ncr$ denotes the class of nilpotent groups of nilpotency class at most $c$ generated by at most $r$ elements.
\begin{itemize}
	\item \label{Prob:WP} The \emph{word problem}: given $G \in \Ncr$ and $g\in G$, is $g=1$ in $G$?
	\item \label{Prob:NF} Given $G \in \Ncr$ and $g\in G$, compute the (Mal'cev) normal form of $g$.
	\item \label{Prob:MP} The \emph{subgroup membership problem}: Given $G \in \Ncr$ and $g,h_{1},\ldots,h_{n}\in G$, decide whether
	$g\in \langle h_{1},\ldots,h_{n}\rangle$ and, if so, express $g$ as a word over the subgroup generators $h_{1},\ldots,h_{n}$  (in \cite{MacDonaldMNV15} only the decision version was shown to be in \L~-- for expressing $g$ as a word over the original subgroup generators a polynomial time bound was given).
	\item \label{Prob:KP} 
	Given $G,H \in \Ncr$ and $K=\langle g_{1},\ldots,g_{n}\rangle \leq G$, together with a homomorphism $\phi:K\rightarrow H$ specified 
	by $\phi(g_{i})=h_{i}$, and some $h\in \mathrm{Im}(\phi)$, compute a generating set for $\ker(\phi)$ and find $g\in G$ such that $\phi(g)=h$.
	\item \label{Prob:P} Given $G \in \Ncr$ and $K=\langle g_{1},\ldots,g_{n}\rangle \leq G$, compute a presentation for $K$.
	\item \label{Prob:C} Given $G \in \Ncr$ and $g\in G$, compute a generating set for the centralizer of $g$.
	\item \label{Prob:CP} The \emph{conjugacy problem}: Given $G \in \Ncr$ and $g,h\in G$, decide whether or not there exists $u\in G$ such that $u^{-1}gu=h$ and, if so, find such an element $u$.
\end{itemize}
These problems are not only of interest in themselves, but also might serve as building blocks for solving the same problems in polycyclic groups~-- which are of particular interest because of their possible application in non-commutative cryptography \cite{EickK04}.
In this work we follow \cite{MacDonaldMNV15} and extend these results in several ways:
\begin{itemize}
	\item We give a complexity bound of uniform \TC for all the above problems.
	\item In order to derive this bound, we show that the extended gcd problem (given $a_1, \dots, a_n \in \Z$, compute $x_1, \dots, x_n \in \Z$ with $\gcd(a_1, \dots, a_n) = \sum_i a_ix_i$) with input and output in unary is in uniform \TC.
	\item Our description of circuits is for the uniform setting where $G \in \Ncr$ is part of the input (in \cite{MacDonaldMNV15} the uniform setting is also considered; however, only in some short remarks).
	\item  Since nilpotent groups have polynomial growth, it is natural to allow compressed inputs: we give a uniform \TC solution for the word problem allowing words with binary exponents as input~-- this contrasts with the situation with straight-line programs (\ie context-free grammars which produces precisely one word~-- another method of exponential compression) as input: then the word problem is hard for $\mathsf{C}_=\mathsf{L}$ \cite{KoenigL15}. Thus, the difficulty of the word problem with straight-line programs is not due to their compression but rather due to the difficulty of evaluating a straight-line program.
	\item We show that the other of the above problems are uniform-\TC-Turing-reducible to the (binary) extended gcd problem  
	 when the inputs (both the ambient group and the subgroup etc.) are given as words with binary exponents.
	\item We show how to solve the \emph{power problem} in nilpotent groups. This allows us to apply a result from \cite{MiasnikovVW17} in order to show that iterated wreath products of nilpotent groups have conjugacy problem in uniform \TC.
\end{itemize}
Thus, in the unary case we settle the complexity of the above problems completely. Moreover, it also seems rather difficult to solve the subgroup membership problem without computing gcds~-- in this case our results on binary inputs would be also optimal. Altogether, our results mean that many algorithmic problems are  no more complicated in nilpotent groups than in abelian groups.
Notice that while in \cite{MacDonaldMNV15} explicit length bounds on the outputs for all these problems are proven, we obtain polynomial length bounds simply by the fact that everything can be computed in uniform \TC (for which in the following we only write \TC).

Throughout the paper we follow the outline of \cite{MacDonaldMNV15}.
For a concise presentation, we copy many definitions from \cite{MacDonaldMNV15}.
Most of our theorems involve two statements: one for unary encoded inputs and one for binary encoded inputs. In order to have a concise presentation, we always put them in one result. We only consider finitely generated nilpotent groups without mentioning that further.

\vspace{-2mm}
\subparagraph*{Outline.}

We start with basic definitions on complexity as well as on nilpotent groups. In \prettyref{sec:subgroup} we describe how subgroups of nilpotent groups can be represented and develop a ``nice'' presentation for all groups in \Ncr. \prettyref{sec:WP} deals with the word problem and computation of normal forms. After that we solve the unary extended gcd problem in \TC and introduce the so-called matrix reduction in order to solve the subgroup membership problem. In \prettyref{sec:more} we present our result for the remaining of the above problems, in \prettyref{sec:Uniform2} we explain how to compute ``nice'' presentations, and in \prettyref{sec:pp} we apply the results of \cite{MiasnikovVW17} in order to show that the conjugacy problem of iterated wreath products of nilpotent groups is in \TC.
Finally, we conclude with some open questions.

\section{Preliminaries}
\subsection{Complexity}\label{sec:complexity}

For a finite \emph{alphabet} $\Sig$, the set of \emph{words} over $\Sig$ is denoted by $\Sig^*$.
Computation or decision problems are given by functions $f:\Del^* \to \Sig^*$ for some finite alphabets $\Del $ and $\Sig$. A decision problem ($=$ formal language) $L$ is identified with its characteristic function $\chi_L: \Del^* \to \oneset{0,1}$ with $\chi_L(x)=1$ if, and only if, $x \in L$. (In particular, the word and conjugacy problems can be seen as functions $\Sig^* \to \oneset{0,1}$.) We use circuit complexity as described in \cite{Vollmer99}.

\vspace{-3mm}\subparagraph*{Circuit Classes.}
The class \Tc{0} is defined as the class of functions computed by families of circuits of constant depth and polynomial size with unbounded fan-in Boolean gates (and, or, not) and majority gates. A majority gate (denoted by $\mathrm{Maj}$) returns $1$ if the number of $1$s in its input is greater or equal to the number of $0$s. In the following we always assume that the alphabets $\Del$ and $\Sig$ are encoded over the binary alphabet $\oneset{0,1}$ such that each letter uses the same number of bits. 
We say a function $f$ is \emph{\TC-computable} if $f \in \TC$.

In the following, we only consider $\mathsf{Dlogtime}$-uniform circuit families and we simply write $\TC$ as shorthand for $\mathsf{Dlogtime}$-uniform $\TC$. $\mathsf{Dlogtime}$-uniform means that there is a deterministic Turing machine which decides in time $\Oh(\log n)$ on input of two gate numbers (given in binary) and the string $1^n$ whether there is a wire between the two gates in the $n$-input circuit and also computes of which type some gates is. Note that the binary encoding of the gate numbers requires only $\Oh(\log n)$ bits~-- thus, the Turing machine is allowed to use time linear in the length of the encodings of the gates.
For more details on these definitions we refer to \cite{Vollmer99}.

We have the following inclusions (note that even $\TC \sse \P$ is not known to be strict):
\begin{align*}
\Ac0 \ssnq \TC \sse \L \sse \P.
\end{align*}

\vspace{-3mm}\subparagraph*{Reductions.} 
A function $f$ is \emph{\TC-Turing-reducible} to a function $g$ if there is a $\mathsf{Dlogtime}$-uniform family of \TC circuits computing $f$ which, in addition to the Boolean and majority gates, also may use oracle gates for $g$ (\ie gates which on input $x$ output $g(x)$). This is expressed by $f \in \TC(g)$. 
Note that if $f_1, \dots, f_k$ are in \TC, then $\TC(f_1, \dots, f_k) = \TC$. 

In particular, if $f$ and $g$ are $\TC$-computable functions, then also the composition $g \circ f $ is \TC-computable. 
We will extensively make use of this observation~-- which will also guarantee the polynomial size bound on the outputs of our circuits without additional calculations.

We will also use another fact frequently without giving further reference: on input of two alphabets $\Sig$ and $\Del$ (coded over the binary alphabet), a list of pairs $(a,v_a)$ with $a \in \Sig$ and $v_a \in \Del^*$ such that each $a \in \Sig$ occurs in precisely one pair, and a word $w \in \Sig^*$, the image $\phi(w)$ under the homomorphism $\phi$ defined by $\phi(a) = v_a$ can be computed in \TC \cite{LangeM98}.

\vspace{-3mm}\subparagraph*{Encoding numbers: unary vs.\ binary.}
There are essentially two ways of representing integer numbers: the usual way as a binary number where a string $a_0 \cdots a_n$ with $a_i \in \oneset{0,1}$ represents $\sum a_i 2^{n - i}$, and as a unary number where $k \in \N$ is represented by $1^k = \smash{\underbrace{11\cdots 1}_{k}}$ (respectively by $0^{n-k} 1^k$ if $n$ is the number of input bits).

We will state most results in this paper with both representations. The unary representation corresponds to group elements given as words over the generators, whereas the binary encoding will be used if inputs are given in a compressed form.

\begin{example}\label{ex:countTC} 
	The following problem $\mathrm{Count}$ is in \TC: given a bit-string $u$ of length $n$ and a number $j<n$ (we assume that it is given in unary as $0^{n-j}1^j$), decide whether the number of ones $\abs{u}_1$ in $u$ is exactly $j$.
	We have $\abs{u}_1 \geq j$ if, and only if, $\abs{u0^{j}1^{n-j}}_1 \geq n$. Thus,	
	\[\mathrm{Count}(u,j) = \mathrm{Maj}(u0^{j}1^{n-j})  \land  \left( \lnot \mathrm{Maj}(u0^{j}1^{n-j})  \right).\]
	In particular, the word problem of $\Z$ when $1$ is encoded as $1$ and $-1$ as $0$, which is simply the question whether $\abs{u}_1= n/2$ and $n$ even, is in \TC. 
\end{example}

\vspace{-3mm}\subparagraph*{Arithmetic in \TC.}
\proc{Iterated Addition} (resp.\ \proc{Iterated Multiplication}) are the following computation problems: On input of $n$ binary integers $a_1,\dots, a_n$ each having $n$ bits (\ie the input length is $N=n^2$), compute the binary representation of the sum $\sum_{i=1}^n a_i$ (resp.\ product $\prod_{i=1}^n a_i$).
For \proc{Integer Division} the input are two binary $n$-bit integers $a, b$; the binary representation of the integer $c=\floor{a/b}$ has to be computed.
The first statement of Theorem~\ref{thm:divisionTC} is a standard fact, see \cite{Vollmer99}; the other statements are due to Hesse, \cite{hesse01, HeAlBa02}.
\begin{theorem}[\cite{hesse01, HeAlBa02, Vollmer99}]\label{thm:divisionTC}\label{thm:additionTC}
	The problems \proc{Iterated Addition},
	\proc{Iterated Multiplication},
	\proc{Integer Division} are all in \TC no matter whether inputs are given in unary or binary.
\end{theorem}
Note that if the numbers $a$ and $b$ are encoded in unary (as strings $1^a$ and $1^b$), division can be seen to be in \TC very easily: just try for all $0\leq c \leq a$ whether $0 \leq a - bc < b$.

\vspace{-3mm}\subparagraph*{Representing groups for algorithmic problems.}
We consider finitely generated groups $G$ together with finite generating sets $A$. Group elements are represented as words over the generators and their inverses (\ie as elements of $(A \cup A^{-1})^*$). 
We make no distinction between words and the group elements they represent. Whenever it might be unclear whether we mean equality of words or of group elements, we write ``$g=h$ in $G$'' for equality in $G$.

Words over the generators $\pm1$ of $\Z$ correspond to unary representation of integers. As a generalization of binary encoded integers, we introduce the following notion: a \emph{word with binary exponents} is a sequence $w_1, \dots, w_n$ where the $w_i$ are from a fixed generating set of the group together with a sequence of exponents $x_1, \dots, x_n$ where the $x_i \in \Z$ are encoded in binary. The word with binary exponents represents the word (or group element) $w=w_1^{x_1} \cdots w_n^{x_n}$.
Note that in a fixed nilpotent group \emph{every} word of length $n$ can be rewritten as a word with binary exponents using $\Oh(\log n)$ bits (this fact is well-known and also a consequence of Theorem~\ref{thm:nilpotentExpWP} below); thus, words with binary exponents are a natural way of representing inputs for algorithmic problems in nilpotent groups.

\subsection{Nilpotent groups and Mal'cev coordinates}\label{Section:Nilpotent}\label{sec:nilpotent}

Let $G$ be a group. For $x,y\in G$ we write $x^y=y^{-1} x y$ ($x$ \emph{conjugated} by $y$) and $[x,y] = x^{-1} y^{-1} xy$ (\emph{commutator} of $x$ and $y$). 
For subgroups $H_1, H_2 \leq G$, we have $[H_1,H_2] = \gen{\set{[h_1,h_2]}{ h_1 \in H_1, h_2 \in H_2}}$.
A group $G$ is called \emph{nilpotent} if it has central series, i.e. 
\begin{equation}\label{Eqn:CyclicSeries}
G=G_{1}\geq G_{2}\geq \cdots \geq G_{c}\geq G_{c+1}=1
\end{equation}
such that $[G,G_{i}]\leq G_{i+1}$ for all $i=1,\ldots,c$.
If $G$ is finitely generated, so are the abelian quotients $G_i/G_{i+1}$, $1\le i\le c$. Let $a_{i1},\ldots,a_{im_i}$ be a basis of $G_i/G_{i+1}$, i.e. a generating set such that $G_i/G_{i+1}$ has a presentation $\genr{ a_{i1},\ldots,a_{im_i}} {\!a_{ij}^{e_{ij}},\:\![a_{ik},a_{i\ell}],\text{ for } j\in \mathcal T_i,\, k,\ell \in \{1,\ldots,m_{i}\}\!}$, where $\mathcal{T}_{i}\subseteq\{1,\ldots,m_{i}\}$ (here $\cT$ stands for torsion) and $e_{ij}\in\Z_{>0}$ (be aware that we explicitly allow $e_{ij}=1$, which is necessary for our definition of \qmps in Section~\ref{sec:subgroup}). Formally put $e_{ij}=\infty$ for $j\notin \mathcal T_i$. Note that
\[
A=(a_{11},a_{12},\ldots,a_{cm_c})
\]
is a so-called polycyclic generating sequence for $G$, and we call $A$ a {\em Mal'cev basis associated to the central series}~\eqref{Eqn:CyclicSeries}. Sometimes we use $A$ interchangeably also for the set $A=\oneset{a_{11},a_{12},\ldots,a_{cm_c}}$.

For convenience, we will also use a simplified notation, in which the generators $a_{ij}$ and exponents $e_{ij}$ are renumbered by replacing each subscript $ij$ with $j+\sum\limits_{\ell<j}m_\ell$, so the generating sequence $A$ can be written as $A=(a_1,\ldots, a_m)$. We allow the expression $ij$ to stand for  
$j+\sum\limits_{\ell<j}m_\ell$ in other notations as well.
We also denote
\[
\mathcal{T} = \{i \;|\; e_{i}<\infty\}.
\]
By the choice of  $\{a_{1},\ldots,a_{m}\}$, every element $g\in G$ may be written uniquely in 
the form
\[
g=a_{1}^{\alpha_{1}}\cdots a_{m}^{\alpha_{m}},
\]
where $\alpha_{i}\in \Z$ and $0\leq \alpha_{i}<e_{i}$ whenever $i\in \mathcal{T}$.
The $m$-tuple $(\alpha_{1},\ldots,\alpha_{m})$ is called the \emph{coordinate vector} or \emph{Mal'cev coordinates}
of $g$ and is denoted $\coords{g}$, and the expression $a_{1}^{\alpha_{1}}\cdots a_{m}^{\alpha_{m}}$ is called the \emph{(Mal'cev) normal form} of $g$. We also 
denote $\alpha_{i}=\coordsj{g}{i}$.

To a Mal'cev basis $A$ we associate a presentation of $G$ as follows. 
For each $1\le i\le m$, let $n_i$ be such that $a_i\in G_{n_i}\setminus G_{n_i+1}$. If $i\in \mathcal{T}$, then $a_{i}^{e_{i}}\in G_{n_i+1}$, hence a relation 
\begin{align}\label{stdpolycyclic1}
a_{i}^{e_{i}} = a_{\ell}^{\mu_{i\ell}}\cdots a_{m}^{\mu_{im}}
\end{align}
holds in $G$ for $\mu_{ij}\in\Z$ and $\ell>i$ such that $a_\ell,\ldots,a_m\in G_{n_i+1}$.  Let $1\leq i<j\leq m$.  Since the series~\eqref{Eqn:CyclicSeries} is central, 
relations of the form
\begin{align}\allowdisplaybreaks
a_{j}a_{i} & =   a_{i} a_{j}a_{\ell}^{\alpha_{ij\ell}}\cdots a_{m}^{\alpha_{ijm}} \label{stdpolycyclic2}\\
a_{j}^{-1}a_{i} & =  a_{i} a_{j}^{-1} a_{\ell}^{\beta_{ij\ell}}\cdots a_{m}^{\beta_{ijm}} \label{stdpolycyclic3}
\end{align}
hold in $G$ for $\alpha_{ijk},\beta_{ijk}\in\Z$ and $l>j$ such that $a_\ell,\ldots, a_m\in G_{n_j+1}$. 
Now, $G$ is the group with generators $\{a_{1},\ldots,a_{m}\}$ subject to the relation of the the form (\ref{stdpolycyclic1})--(\ref{stdpolycyclic3}). 

A presentation with relations of the form (\ref{stdpolycyclic1})--(\ref{stdpolycyclic3}) for all $i$ resp.\ $i$ and $j$ is called a \emph{nilpotent presentation}.
Indeed, any presentation of this form will define a nilpotent group.
It is called \emph{consistent} if the order of $a_{i}$ 
modulo $\gen{ a_{i+1},\ldots,a_{m}}$ is precisely $e_{i}$ for all $i$. While presentations of this 
form need not, in general, be consistent, those derived from a central series of a group $G$ as above are consistent. 

Given a consistent nilpotent presentation, there is an easy way to solve the word problem: simply apply the rules of the form (\ref{stdpolycyclic2}) and (\ref{stdpolycyclic3}) to move all occurrences of $a_1^{\pm 1}$ in the input word to the left, then apply the power relations (\ref{stdpolycyclic1}) to reduce their number modulo $e_1$; finally, continue with $a_2$ and so on.

\vspace{-3mm}\subparagraph*{Multiplication functions.}  An crucial feature of the coordinate vectors for nilpotent groups is that the coordinates of a product $(a_{1}^{\alpha_{1}}\cdots a_{m}^{\alpha_{m}})(a_{1}^{\beta_{1}}\cdots a_{m}^{\beta_{m}})$ may be computed as a ``nice''  function (polynomial if $\cT = \emptyset$) of the integers $\alpha_{1},\ldots,\alpha_{m},\beta_{1},\ldots,\beta_{m}$. 

\begin{lemma}[\cite{Hal69, Kargapolov-Merzlyakov}]\label{lem:malcev}\label{lem:ab_operation}
	Let $G$ be a nilpotent group with Mal'cev basis $a_{1},\ldots,a_{m}$ and $\cT = \emptyset$. 
	There exist $p_{1},\ldots,p_{m}\in \Z[x_1, \dots, x_m,y_1,\dots,y_m]$ and 
	$q_{1},\ldots,q_{m}\in\Z[x_1, \dots, x_m,z]$ such that for $g,h\in G$ with $\coords{g}=(\gamma_{1},\ldots,\gamma_{m})$ and $\coords{h}=(\delta_{1},\ldots,\delta_{m})$ and  $l\in\Z$ we have
	\begin{enumerate}
		\item $\coordsj{gh}{i} = p_{i}(\gamma_{1},\ldots,\gamma_{m},\delta_{1},\ldots,\delta_{m})$,
		\item $\coordsj{g^{l}}{i} = q_{i}(\gamma_{1},\ldots,\gamma_{m}, l)$,
		\item $\coordsj{gh}{1}=\gamma_{1}+\delta_{1}$ and $\coordsj{g^{l}}{1}=l \gamma_{1}$.
	\end{enumerate}
\end{lemma}
Notice that an explicit algorithm to construct the polynomials $p_i,q_i$ is given in \cite{LGS98}. For further background on nilpotent groups we refer to \cite{Hal69, Kargapolov-Merzlyakov}.

\section{Presentation of subgroups}\label{sec:subgroup}

Before we start with algorithmic problems, we introduce a canonical way how to represent subgroups of nilpotent groups. This is important for two reasons: first, of course we need it to solve the subgroup membership problem, and, second, for the uniform setting it allows us to represent nilpotent groups as free nilpotent group modulo a kernel which is represented as a subgroup. Let $h_{1},\ldots, h_{n}$ be elements of $G$ given in normal form by 
$h_{i} = a_{1}^{\alpha_{i1}}\cdots a_{m}^{\alpha_{im}}$, 
for $i=1,\ldots,n$, and let $H=\gen{ h_{1},\ldots,h_{n}}$.
We associate the \emph{matrix of coordinates}
\begin{equation}\label{Eqn:CoordinateMatrix}
A=
\left(\begin{array}{ccc}
\alpha_{11} &\cdots & \alpha_{1m} \\
\vdots&\ddots & \vdots\\
\alpha_{n1} & \cdots & \alpha_{nm}
\end{array}\right),
\end{equation}
to the tuple $(h_{1},\ldots,h_{n})$ and conversely, to any $n\times m$ integer matrix, we associate an $n$-tuple of elements of $G$, whose Mal'cev coordinates are given as the rows of the matrix, and the subgroup $H$ generated 
by the tuple.
For each $i=1,\ldots,n$ where row $i$ is non-zero, let $\pivot{i}$ be the column of the first non-zero entry (`pivot') in row $i$.
The sequence $(h_{1},\ldots,h_{n})$ is said to be in \emph{standard form} if the matrix of coordinates $A$ is in row-echelon form and its pivot columns are maximally reduced (similar to the Hermite normal form), more specifically, if $A$ satisfies the following properties:
\begin{enumerate}
	\item all rows of $A$ are non-zero (i.e. no $h_{i}$ is trivial),\label{li:std_nontrivial}
	\item $\pivot{1} < \pivot{2} < \cdots < \pivot{s}$ (where $s$ is the number of pivots),\label{li:std_echelon}
	\item $\alpha_{i \pivot{i}}>0$, for all $i=1,\ldots,n$,\label{li:std_positive}
	\item $0\leq \alpha_{k \pivot{i}} < \alpha_{i \pivot{i}}$, for all $1\leq k < i\leq s$ \label{li:std_reduced}
	\item  if $\pivot{i}\in\mathcal{T}$, then $\alpha_{i \pivot{i}}$ divides $e_{\pivot{i}}$, for $i=1,\ldots,s$.\label{li:std_torsion}
\end{enumerate}
The sequence (resp.\ matrix) is called {\em full} if in addition 
\begin{enumerate}
	\setcounter{enumi}{5}
	\item \label{li:std_full}  
	$H\cap \gen{ a_i,a_{i+1},\ldots, a_m}$ is generated by $\{h_{j}\mid \pi_j\ge i\}$, 
	for all $1\le i\le m$.
\end{enumerate}
Note that $\{h_{j}\mid \pi_j\ge i\}$ consists of those elements 
having 0 in their first $i-1$ coordinates.  
It is an easy exercise (see also \cite{MacDonaldMNV15}) to show that \ref{li:std_full} 
holds for a given $i$ if, and only if,
\begin{itemize}
	\item for all $1\leq k<j\leq s$ with $\pivot{k}<i$, 
	$h_{k}^{-1} h_{j} h_{k}$ and $h_{k} h_{j} h_{k}^{-1}$ are 
	elements of $\genr{ h_{l}}{l>k}$, and
	\item for all $1\leq k\leq s$ with $\pivot{k}<i$ and $\pivot{k}\in\mathcal{T}$, 
	$h_{k}^{e_{\pivot{k}}/\alpha_{k\pivot{k}}}\in \genr{ h_{l}}{ l>k}$.
\end{itemize}
We will use full sequences and the associated matrices in full form interchangeably without mentioning it explicitly. For simplicity we assume that the inputs of algorithms are given as matrices.
The importance of full sequences is described in the following lemma~-- a proof can be found in \cite{Sim94} Propositions 9.5.2 and 9.5.3.
\begin{lemma}[{\cite[Lem.\ 3.1]{MacDonaldMNV15}}]
	\label{lem:UniqueStandardForm}
	Let $H\leq G$.  There is a unique full sequence $U=(h_{1},\ldots,h_{s})$ that generates 
	$H$.  We have $s\leq m$ and 
	$
	H=\{h_{1}^{\beta_{1}}\cdots h_{s}^{\beta_{s}} \,|\, \beta_{i}\in\Z \mbox{ and $0\leq \beta_{i}<e_{\pivot{i}}$ if $\pivot{i}\in\mathcal{T}$}\}.
	$
\end{lemma}
Thus, computing a full sequence will be the essential tool for solving the subgroup membership problem. Before we focus on subgroup membership, we will first solve the word problem and introduce how the nilpotent group can be part of the input.

\subsection{Quotient presentations}
\label{Section:Uniform1}
\label{sec:presentation}

Let $c,r\in \N$ be fixed. The free nilpotent group $F_{c,r}$  of class $c$ and rank $r$ is defined as
$F_{c,r} = \genr{a_1, \dots, a_r}{[x_1, \dots , x_{c+1}] = 1 \text{ for } x_1, \dots,x_{c+1} \in F_{c,r} }$ where  $[x_1, \dots , x_{c+1}] = [[x_1, \dots , x_{c}],x_{c+1}]$, \ie $F_{c,r}$ is the $r$-generated group only subject to the relations that weight $c+1$ commutators are trivial. 
Throughout, we fix a Mal'cev basis $A = (a_1, \dots, a_m)$ (which we call the \emph{standard Mal'cev basis}) associated to the lower central series of $F_{c,r}$ such that the associated nilpotent presentation consists only of relations of the form (\ref{stdpolycyclic2}) and (\ref{stdpolycyclic3}) (\ie $\cT=\emptyset$~-- such a presentation exists since $F_{c,r}$ is torsion-free), $a_1, \dots, a_r$ generates $F_{c,r}$, and all other Mal'cev generators are iterated commutators of $a_1, \dots, a_r$. 

Denote by $\Ncr$ the set of $r$-generated nilpotent groups of class at most~$c$. 
Every group $G\in \Ncr$ is a quotient of the free nilpotent group $F_{c,r}$, \ie $G= F_{c,r}/N$ for some normal subgroup $N \leq F_{c,r}$.
Assume that $T=(h_{1},\ldots,h_{s})$ is a full sequence generating $N$. Adding $T$ to the set of relators of the free nilpotent group yields a new nilpotent presentation. This presentation will be called \emph{\qmp} of $G$. For inputs of algorithms, we assume that a \qmp is always given as its matrix of coordinates in full form. Depending whether the entries of the matrix are encoded in unary or binary, we call the \qmp be given in \emph{unary} or \emph{binary}.

\begin{lemma}[{\cite[Prop.\ 5.1]{MacDonaldMNV15}}]\label{lem:quotient_presentation}
	Let $c$ and $r$ be fixed integers and let $A = (a_1, \dots, a_m)$ be the standard Mal'cev basis of $F_{c,r}$. Moreover, denote by $S$ the set of relators of $F_{c,r}$ with respect to $A$. 
	Let $G \in \Ncr$ with $G = F_{c,r}/N$ and let $T$ be the full-form sequence for the subgroup $N$ of $F_{c,r}$. Then, $\genr{A}{S \cup T}$ is a consistent nilpotent presentation of $G$.
\end{lemma}

\begin{proof}
	Clearly, we have $G\simeq \langle A\mid S\cup T\rangle$.  Since $\langle A\mid S\rangle$ is a nilpotent 
	presentation and the elements of $T$ add relators of the form (\ref{stdpolycyclic1}), 
	the presentation is nilpotent. To prove that it is consistent, suppose some 
	$a_i\in A$ has order $\alpha_i$ modulo $\langle a_{i+1},\ldots,a_m\rangle$ in 
	$\langle A\mid S\cup T\rangle$. Since the order is infinite in $F$, there must 
	be element of the form $a_i^{\alpha_i} a_{i+1}^{\alpha_{i+1}}\cdots a_m^{\alpha_m}$ in 
	$N$. But then, by \prettyref{lem:UniqueStandardForm}, $T$ must contain an element $a_i^{\alpha'_i} a_{i+1}^{\alpha'_{i+1}}\cdots a_m^{\alpha'_m}$ where $\alpha'_i$ divides $\alpha_i$.  Hence $\alpha_i$ cannot be smaller than $\alpha'_i$ 
	and so the presentation is consistent.
\end{proof}

 For the following we always assume that a \qmp is part of the input, but $c$ and $r$ are fixed. Later, we will show how to compute \qmps from an arbitrary presentation.

\begin{remark}
	Lemma~\ref{lem:quotient_presentation} ensures that each group element has a unique normal form with respect to the \qmp; thus, it guarantees that all our manipulations of Mal'cev coordinates are well-defined.
\end{remark}

\section{Word problem and computation of Mal'cev coordinates}\label{sec:WP}

In this section we deal with the word problem of nilpotent groups, which is well-known to be in \TC \cite{Robinson93phd}. Here, we generalize this result by allowing words with binary exponents (recall that \emph{word with binary exponents} is a sequence $w=w_1^{x_1} \cdots w_n^{x_n}$ where $w_i \in \oneset{a_1, \dots, a_m}$  and the $x_i \in \Z$). 
By using words with binary exponents the input can be compressed exponentially~-- making the word problem, a priori, harder to solve.
Nevertheless, it turns out that the word problem still can be solved in \TC when allowing the input to be given as a word with binary exponents. Note that this contrasts with the situation where the input is given as straight-line program (which like words with binary exponents allow an exponential compression)~-- then the word problem is complete for the counting class $\mathsf{C}_=\mathsf{L}$ \cite{KoenigL15}.

\begin{theorem}\label{thm:nilpotentExpWP}
	Let $c,r  \geq 1 $ be fixed and let $(a_1, \dots, a_m)$ be the standard Mal'cev basis of $F_{c,r}$. The following problem is \TC-complete: on input of
	 \begin{itemize}
		\item 
	$G \in \Ncr$ given as a binary encoded \qmp and
		\item 
	a word with binary exponents $w=w_1^{x_1} \cdots w_n^{x_n}$,
		\end{itemize}
	compute integers $y_1, \dots, y_m$ (in binary) such that $w = a_1^{y_1} \cdots a_m^{y_m}$ in $G$ and $0\leq y_i < e_i$ for $i \in \cT$. 
	Moreover, if the input is given in unary (both $G$ and $w$), then the output is in unary.
\end{theorem}
Note that the statement for unary inputs is essentially the one of \cite{Robinson93phd}.
Be aware that in the formulation of the theorem, $\cT$ and $e_i$ for $i\in \cT$ depend on the input group $G$. These parameters can be read from the full matrix $ (\alpha_{ij})_{i,j}$ of coordinates representing $G$ (recall that $\pi_i$ denotes the column index of the $i$-th pivot and here $s$ is the number of rows of the matrix):
 \[\cT = \set{\pi_i}{i\in \oneset{1, \dots, s}}\] 
 (all columns which have a pivot)  and $e_i = \alpha_{ji}$ if $\pi_j = i$. 
As an immediate consequence of Theorem~\ref{thm:nilpotentExpWP}, we obtain:
\begin{corollary}\label{cor:nilpotentExpWP}
	Let $c,r \geq 1  $ be fixed. 
	The uniform, binary version of the word problem for groups in $\Ncr$ is \TC-complete (where the input is given as in Theorem~\ref{thm:nilpotentExpWP}). 
\end{corollary}
The proof  of Theorem~\ref{thm:nilpotentExpWP} follows the outline given in Section~\ref{sec:nilpotent}; however, we cannot apply the rules (\ref{stdpolycyclic1})--(\ref{stdpolycyclic3}) one by one. Instead we make only two steps for each generator: first apply all possible rules (\ref{stdpolycyclic2}) and (\ref{stdpolycyclic3}) in one step and then apply the rules (\ref{stdpolycyclic1}) in one step.
\begin{proof}[Proof of Theorem~\ref{thm:nilpotentExpWP}]	
	The hardness part is clear since already the word problem of $\Z$ is \TC-complete.
	For describing a \TC circuit, we proceed by induction along the standard Mal'cev basis $(a_1, \dots, a_m)$ of the free nilpotent group $F_{c,r}$. If $w$ does not contain any letter $a_1$, we have $y_1 = 0$ and we can compute $y_i$ for $i>1$ by induction.
	
	Otherwise, we rewrite $w $ as $ a_1^{y_1} uv$ (with $0\leq y_1 < e_1$ if $1 \in \cT$) such that $u$ and $v$ are words with binary exponents not containing any $a_1$s. Once this is completed, $uv$ can be rewritten as $a_2^{y_2} \cdots a_m^{y_m}$ by induction. 
	For computing $y_1$, $u$ and $v$, we proceed in two steps:
	
	First, we rewrite $w $ as $ a_1^{\tilde y_1} v$ with $\tilde y_1 = \sum_{w_i = a_1} x_i$
	(this is possible by Lemma~\ref{lem:ab_operation} (iii)).
	The exponent $\tilde y_1$ can be computed by iterated addition, which by Theorem~\ref{thm:divisionTC} is in \TC (in the unary case $\tilde y_1$ can be written down in unary). Now, $v$ consists of what remains from $w$ after $a_1$ has been ``eliminated'': for every position $i$ in $w$ with $w_i \neq a_1$, we compute $z_i = \sum_{\stackrel{j>i}{w_j = a_1}}  x_j$ using iterated addition. Let $w_i = a_k$. By Lemma~\ref{lem:malcev} (i) there are fixed polynomials $p_{k,k+1}, \dots, p_{k,m} \in \Z[x,y]$  such that in the free nilpotent group holds
	\[a_k^{x}a_1^y = a_1^y a_k^x\, a_{k+1}^{p_{k,k+1}(x,y)} \cdots a_{m}^{p_{k,m}(x,y)} \qquad\text{ for all } x,y \in \Z. \]
	Hence, in order to obtain $\tilde w$, it remains to replace every $w_i^{x_i}$ with $w_i = a_1$ by the empty word and every $w_i^{x_i}$ with $w_i= a_k \neq a_1$ by $ a_k^{x_i} a_{k+1}^{p_{k,k+1}(x_i,z_i)} \cdots a_{m}^{p_{k,m}(x_i,z_i)}$, which is a word with binary exponents  (resp.\ as a word  of polynomial length in the unary case), for $k= 2, \dots, m$. The exponents can be computed in \TC by Theorem~\ref{thm:divisionTC}. Since the $p_{k,i}$ are bounded by polynomials, in the unary case, $ a_k^{x_i} a_{k+1}^{p_{k,k+1}(x_i,z_i)} \cdots a_{m}^{p_{k,m}(x_i,z_i)}$ can be written as a word without exponents. 
	
	The second step is only applied if $1 \in \cT$ (as explained above, this can be decided and $e_i$ can be read directly from the \qmp by checking whether there is a pivot in the first column)~-- otherwise $y_1= \tilde y_1$ and $u$ is the empty word. We rewrite $a_1^{\tilde y_1}$ to $ a_1^{y_1} u$ with $y_1 = \tilde y_1 \bmod e_1$ and a word with binary exponents $u$ not containing any $a_1$. Again $y_1$ can be computed in \TC by Theorem~\ref{thm:divisionTC}.
	Let $a_{1}^{e_{1}} = a_{2}^{\mu_{12}}\cdots a_{m}^{\mu_{1m}}$ be the power relation for $a_1$ (which can be read from the \qmp~-- it is just the row where the pivot is in the first column) and write $\tilde y_1 = s\cdot e_1 +  y_1$. Now, $u$ should be equal to $(a_{2}^{\mu_{12}}\cdots a_{m}^{\mu_{1m}})^s$ in $F_{c,r}$. 
	We use the fixed polynomials $q_i \in \Z[x_1, \dots, x_m,z]$ from Lemma~\ref{lem:malcev} (ii) for $F_{c,r}$ 
	yielding \[u = a_{2}^{q_2(0,\mu_{12}, \dots ,\mu_{1m},s)}\cdots a_{m}^{q_m(0,\mu_{12}, \dots, \mu_{1m},s)}\] (which, in the binary setting, is a word with binary exponents, and in the unary setting a word without exponents of polynomial length).
	Now, we have $w= a_1^{y_1} uv$ in $G$ as desired.
\end{proof}

\section{The extended gcd problem}\label{sec:gcd}

Computing greatest common divisors and expressing them as a linear combination is an essential step for solving the subgroup membership problem. Indeed, consider the nilpotent group $\Z$ and let $a,b,c \in \Z$. Then $c \in \gen{ a,b}$ if, and only if, $\gcd(a,b) \mid c$.

\subparagraph*{Binary gcds.}
The (binary) \emph{extended gcd problem} (\ExtGCD) is as follows: on input of binary encoded numbers $a_1, \dots, a_n \in \Z$, compute $x_1 , \dots, x_n \in \Z$ such that \[x_{1} a_{1} + \dots + x_{n} a_{n} =  \gcd(a_{1},\ldots,a_{n}).\]	
Clearly this can be done in \P using the Euclidean algorithm, but it is not known whether it is actually in \NC.
Since we need to compute greatest common divisors, we will reduce the subgroup membership problem to the computation of gcds.

\subparagraph*{Unary gcds.}

Computing the $\gcd$ of numbers encoded in unary is straightforward in \TC by an exhaustive search; yet,
it is not obvious how to express $\gcd(a_1, \dots, a_n)$ as $x_{1} a_{1} + \dots + x_{n} a_{n}$ in \TC.
By \cite{MH94} such $x_i$ with $|x_{i}| \leq \frac{1}{2}\max\{|a_{1}|,\ldots, |a_{n}|\}$ can be computed in \L.
However, that algorithm uses a logarithmic number of rounds each depending on the outcome of the previous one~-- so it does not work in \TC. Note that for $n=2$ the problem is easy:

\begin{example}\label{ex:gcdeasy}
	Let $a,b \in \Z$. Then, there are $x,y \in \Z$ with $\abs{x}, \abs{y} \leq \max\oneset{\abs{a}, \abs{b}}$ such that $ax +by = \gcd(a,b)$. 	
	This is easy to see: assume $a,b >0$ (the other cases are similar) and we are given $x,y$ with $ax +by = \gcd(a,b)$ and $x \geq b$, then we can replace $x$ with $x - b$ and $y$ with $y + a$. This does not change the sum and by iterating this step, we can assure that $0\leq x  < b$. Then we have $y= - \frac{ax - \gcd(a,b)}{b}$; hence, $-a < y \leq 1$.
	
	If $a$ and $b$ are given in unary, the coefficients $x,y$ can be computed in \TC by simply checking all (polynomially many) values for $x$ and $y$ with $\abs{x}, \abs{y} \leq \max\oneset{\abs{a}, \abs{b}}$.
\end{example}

However, if we want to express the $\gcd$ of unboundedly many numbers $a_i$ as a linear combination, we cannot check all possible values for $x_1, \dots, x_n$ in \TC because there are $\max\{|a_{1}|^n,\ldots, |a_{n}|^n\}$ (\ie exponentially) many. 
Expressing the gcd as a linear combination can be viewed as a linear equation with integral coefficients. Recently, in \cite[Thm.\ 3.14]{ElberfeldJT11} it has been shown that, if all the coefficients are given in unary, it can be decided in \TC whether such an equation or a system of a fixed number of equations has a solution. 
Since from the proof of \cite[Thm.\ 3.14]{ElberfeldJT11} it is not obvious how to find an actual solution, we prove the following result:

\begin{theorem}\label{thm:GCD}
	The following problem is in \TC:
	Given integers $a_{1},\ldots,a_{n}$ as unary numbers, compute $x_1, \dots, x_n \in \Z$ (either in unary or binary) such that
	\[
	x_{1} a_{1} + \cdots + x_{n} a_{n} =  \gcd(a_{1},\ldots,a_{n})
	\]
	with $|x_{i}| \leq (n+1) \left(\max\{|a_{1}|,\ldots, |a_{n}|\}\right)^2$.
\end{theorem}

\begin{proof}
	Let $A=\max\{|a_{1}|,\ldots, |a_{n}|\}$, which clearly can be computed in \TC. \Wlog we assume that all the $a_i$ are positive.  We assume that all numbers which appear as intermediate results are encoded in binary (indeed, these numbers will grow too fast to encode them in unary).
	
	First observe that $\gcd(a_{1},\ldots,a_{i})$ can be computed in \TC for all $i \in \oneset{1, \dots, n}$. The reason is simply that there are only linearly many numbers less than each $a_i$. In fact, for computing $\gcd(a_{1},\ldots,a_{n})$, the circuit just checks for all $d \leq A$ whether for every $i$ there is some $c_i \leq a_i$ with  $d c_i=a_i$. If for some $d$ there are such $c_i$ for all $i$, we have found a common divisor. The $\gcd$ is simply the largest one. 
	
	Thus, it remains to compute the coefficients $x_i$. Since we can compute $\gcd(a_{1},\ldots,a_{n})$ in \TC, we can divide all numbers $a_i$ by the $\gcd$ and henceforth assume that  $\gcd(a_{1},\ldots,a_{n}) = 1$ (note that this does not change the coefficients $x_i$).	
	
	The first step for computing the $x_i$s, is to compute $d_i = \gcd(a_{1},\ldots,a_{i})$ for $i=1, \dots, n$ and $d_0 = 0$ (note that by our assumption, $d_n = 1$).
	We have $$d_i = \gcd(a_{1},\ldots,a_{i}) = \gcd(\gcd(a_{1},\ldots,a_{i-1}), a_i) = \gcd(d_{i-1}, a_i).$$ Using this observation, the next step computes for each $i$ integers $y_i$ and $z_i$ such that
	$d_i = y_i d_{i-1} + z_i a_i$. For all $i$ this can be done in parallel in \TC by simply trying all possible values with $\abs{y_i}, \abs{z_i} \leq A$ as in \prettyref{ex:gcdeasy}. 
	We set $$x_i = z_i \prod_{j=i+1}^{n} y_j.$$ These $x_i$ can be computed in \TC using iterated multiplication \cite{hesse01}~-- see \prettyref{thm:divisionTC}. Moreover, an easy induction shows that	
	$$x_{1} a_{1} + \cdots + x_{n} a_{n} = \gcd(a_{1},\ldots,a_{n}).$$
	There is only one problem with the numbers $x_i$: in general, they do not meet the  bounds $|x_{i}| \leq (n+1) A^2$. So, the next step will be to modify these $x_i$ in such a way that they meet the desired bound. The idea is to apply a sequence of operations as in \prettyref{ex:gcdeasy} to make the coefficients small. The difficulty here is to find out where exactly to add/subtract a multiple of which $a_i$. 
		
	Let $\cP = \set{i \in \oneset{1, \dots, n}}{x_i > 0}$ and $\cN = \set{i \in \oneset{1, \dots, n}}{x_i  <0}$. Note that $\cP \cap \cN = \emptyset$ and \wwlog we can assume that $\cP \cup \cN = \oneset{1, \dots, n}$.
	For all $i= 1, \dots n$, we set
	\newcommand{\p}{p}
	\newcommand{\n}{n}
	\begin{align}
	\p'_i &= \max\left( 0, \floor{\frac{x_i a_i}{A^2}}\right),&
	\n'_i &= \max\left( 0, \floor{\frac{-x_i a_i}{A^2}}\right).\label{eq:p_prime}
	\end{align}
	Obviously, we have $\p'_i = 0 $ for $i \in \cN$ and  $\n'_i = 0 $ for $i \in \cP$. The non-zero $\p'_i$ correspond to those indices which have a too large positive $x_i$ and the non-zero $\n'_i$ to those indices which have a too small negative $x_i$ (this is because we assumed the $a_i$ to be positive). Moreover, $x_i$ should be decreased (resp. increased) by $A^2\p'_i/a_i$ (resp. $A^2\n'_i/a_i$) in order to make it reasonably small. We will not be able to reach this aim completely, but with a sufficiently small error.
	
	Next, we set $P'_i = \sum_{j=1}^{i}\p'_j$ and $N'_i = \sum_{j=1}^{i}\n'_j$. All the $\p'_i$, $\n'_i$, $P'_i$, $N'_i$ and $\cP$ and $\cN$ can be computed in \TC using iterated addition and division~-- see \prettyref{thm:divisionTC}. 
	
	\begin{lemma}\label{lem:pnbound}
		\[P'_n - N'_n \leq \abs{\cN} \text{ and } N'_n - P'_n \leq \abs{\cP}\]
	\end{lemma}
	\begin{proof}
		For $i\in \cP$, we have $0\leq x_i a_i - p'_i A^2 < A^2$ by definition of $p'_i$. Likewise, we have $0\geq x_i a_i + n'_i A^2 > -A^2$ for $i \in \cN$. Since $\cP \cap \cN = \emptyset$ and $\cP \cup \cN = \oneset{1, \dots, n}$, we obtain \begin{align*}	 
		(P'_n - N'_n)A^2 &= \sum_{i \in \cP}p'_iA^2 - \sum_{i\in \cN} n'_i A^2< \sum_{i \in \cP}x_ia_i + \sum_{i\in \cN} (x_ia_i +  A^2)\\
		&= x_{1} a_{1} + \cdots + x_{n} a_{n} +\abs{\cN}A^2 = 1 +\abs{\cN}A^2
		\end{align*}
	 meaning that $P'_n - N'_n \leq \abs{\cN} $. The same argument yields $(P'_n - N'_n) A^2 > 1 - \abs{\cP} A^2$, and thus $N'_n - P'_n < \abs{\cP}$. 
	\end{proof}
	
	Let $D=N'_n - P'_n$. For $i \in \oneset{1, \dots, n}$, 
	we set
	\begin{align}
	p_i &= \begin{cases}
	p'_i  + 1 & \text{if } i \in \cP \text{ and }i \leq D, \\
	p'_i & \text{otherwise},
	\end{cases} &
	n_i &= \begin{cases}
	n'_i  + 1 & \text{if } i \in \cN \text{ and } i \leq -D, \\
	n'_i & \text{otherwise},\label{eq:p_p_prime}
	\end{cases}
	\end{align}
	and $P_i = \sum_{j=1}^{i}\p_j$ and $N_i = \sum_{j=1}^{i}\n_j$ for $i \in \oneset{0, \dots, n}$. Because of \prettyref{lem:pnbound}, we have $N_n = P_n$. Clearly, the $p_i, n_i, P_i, N_i$ can be computed in \TC and from now on we will work with these numbers. Also, as an immediate consequence of \prettyref{eq:p_prime} and \prettyref{eq:p_p_prime}, we have
	\begin{align}\label{eq:p_p}
	\begin{split}
	-A^2\leq x_ia_i -	p_i A^2 \leq A^2& \qquad\qquad\text{for } i \in \cP, \\
	-A^2\leq x_ia_i +	n_i A^2 \leq A^2& \qquad\qquad\text{for } i \in \cN. 
	\end{split}
	\end{align}
	Now, for $i\in \cP$ and $j \in \cN$, we define
	
		\begin{align*}
	\p_{j,i} = \begin{cases}
	\p_i 			& \text{if }  N_{j-1} \leq P_{i-1} < P_{i} \leq N_{j} \\
	N_j - P_{i-1} 	& \text{if }  N_{j-1} \leq P_{i-1} < N_{j} \leq P_{i} \\
	P_{i} - N_{j-1} & \text{if }  P_{i-1} \leq N_{j-1} < P_{i} \leq N_{j} \\
	\n_j		 	& \text{if }  P_{i-1} \leq N_{j-1} < N_{j} \leq P_{i} \\
	0 & \text{otherwise.}
	\end{cases}
	\end{align*}
	Note that the cases overlap. However, then the different definitions of $\p_{j,i}$ agree.
	For $i\in \cN$ and $j \in \cP$, we set $\p_{j,i} = \p_{i,j}$ and for $i,j \in \cP$ or $i,j \in \cN$ we set $\p_{j,i}= 0$.

\begin{lemma}\label{lem:pij}
We have $\displaystyle\sum_j p_{j,i} = p_i $ and $\displaystyle\sum_i p_{j,i} = n_j $.
\end{lemma}

\begin{proof}
We only show $\sum_j p_{j,i} = p_i $; the other statement follows by symmetry.	First, assume that $ p_i = p_{i,j} $ for some $j$. Then $p_{i,j'} = 0$ for all $j' \neq j$; hence, the lemma holds.
Now, let $ p_i \neq p_{i,j} $ for any $j$. We define
	\begin{align*}
	\alpha_i &= \min \set{j \in \oneset{1, \dots, n}}{P_{i-1} < N_j },&
	\beta_i  &= \max \set{j\in \oneset{1, \dots, n}}{N_{j-1} < P_{i}}.
	\end{align*}
In particular, we have $p_{j,i}=0$ for $j < \alpha_i $ or $j > \beta_i$. Notice that $\alpha_i$ and $\beta_i$ exist for all $i \in \cP$ (since $N_n = P_n$). Also $\alpha_i < \beta_i$ because $\alpha_i = \beta_i =j$ implies $N_{j-1} \leq P_{i-1}< N_j$ and $N_{j-1} < P_{i}\leq N_j$; thus, $p_{j,i} = p_i$. 
Moreover, we have $p_{\alpha_i,i} = N_{\alpha_i} - P_{i-1}$ and $p_{\beta_i,i} =  P_i - N_{\beta_i - 1}$ and $p_{j,i} = n_j$ for $\alpha_i < j < \beta_i$. 
	Since \begin{align*}P_i - P_{i-1} = \sum_{j=0}^{i}p_i -  \sum_{j=0}^{i-1}p_i = p_i\qquad\qquad\qquad \text{ and}\end{align*} 
	\begin{align*} N_{\beta_i-1} -  N_{\alpha_i} - \sum_{j=\alpha_i + 1}^{\beta_i - 1}n_j = \sum_{j=1}^{\beta_i - 1}n_j -  \sum_{j=1}^{\alpha_i}n_j - \sum_{j=\alpha_i + 1}^{\beta_i - 1}n_j = 0,\end{align*} we obtain
	\begin{align*}
\sum_j p_{j,i}  &=N_{\alpha_i} - P_{i-1}   + P_i - N_{\beta_i - 1} + \sum_{j=\alpha_i + 1}^{\beta_i - 1}n_j
= p_i.
\end{align*}
\end{proof}

	We set $y_{j,i} = \floor{\frac{p_{j,i} A^2}{a_i a_j}}$ for $i,j =1, \dots, n$. 
	Notice that, since $a_i a_j\leq A^2$, this means that
	\begin{align}
	(p_{j,i} - 1)A^2 < y_{j,i} a_i a_j \leq p_{j,i} A^2.\label{eq:yaa}
	\end{align}
	Finally, we define our new coefficients $\tilde x_i$ as follows:
	\begin{align*}
	\tilde x_i = \begin{cases}
	x_i - \sum_j y_{j,i} a_j & \text{if } i \in \cP,\\
	x_i + \sum_j y_{i,j} a_j & \text{if } i \in \cN,\\
	x_i& \text{otherwise.}
	\end{cases}
	\end{align*}
	It remains to show the following:
	\begin{enumerate}
		\item the numbers $\tilde x_i$ can be computed in \TC,
		\item $\tilde x_{1} a_{1} + \cdots + \tilde x_{n} a_{n} = 1$,
		\item $\abs{\tilde x_i} \leq (n+1)A^2$ for all $i$.
	\end{enumerate} 
	The first point is straightforward: we already remarked that the $\p_i$, $\n_i$, $P_i$, $N_i$ and $\cP$ and $\cN$ can be computed in \TC. Hence, also the 
	$p_{j,i}$ can be computed in \TC (as simple Boolean combination resp.\ addition of the previous numbers). Now, the $y_{j,i}$ can be computed using division \cite{hesse01}. 
	Finally, the computation of the $\tilde{x_i}$ is simply another application of iterated addition.
	
	For the second point observe that 
	\begin{align*}
	\tilde x_{1} a_{1} + \cdots + \tilde x_{n} a_{n} &= \sum_{i \in \cP}  \tilde{x}_i a_i + \sum_{i \in \cN} \tilde{x}_i a_i
	\\
	&= \sum_{i \in \cP} \left(	x_i - \sum_j y_{j,i} a_j\right) a_i  +\sum_{i \in \cN} \left(	x_i + \sum_j y_{i,j} a_j\right) a_i\\
	&= \sum_{i=1}^n x_i a_i - \sum_{i \in \cP} \sum_j y_{j,i} a_j a_i  +\sum_{i \in \cN} \sum_j y_{i,j} a_j a_i 
	\\
	&= \sum_{i=1}^n x_i a_i - \sum_{i \in \cP} \sum_{j\in \cN} y_{j,i} a_j a_i  +\sum_{i \in \cN} \sum_{j\in \cP} y_{i,j} a_j a_i 
	\\
	&= \sum_{i=1}^n x_i a_i
	\end{align*}
	The last equality is due to the fact that $y_{j,i} = y_{i,j}$ for all $i,j$ and that $y_{i,j} = 0$ if $i$ and $j$ are both in  $\cP$ or both in $\cN$.

	For the third point, let $i \in \cP$. Then, 
	\begin{align*}
	\tilde x_i a_i &= x_ia_i - \sum_j y_{j,i} a_j a_i	\geq  x_ia_i -  \sum_j p_{j,i}A^2 \tag{by \prettyref{eq:yaa}}\\
	&= x_ia_i - p_i A^2\tag{by \prettyref{lem:pij}}\\ 
	&\geq -A^2 \vphantom{A^{A^1}}\tag{by \prettyref{eq:p_p}}
	\intertext{and}
	\tilde x_i a_i &= x_ia_i - \sum_j y_{j,i} a_j a_i 
	\leq  x_ia_i -  \sum_j (p_{j,i} - 1)A^2 \tag{by \prettyref{eq:yaa}}\\
	 &=  x_i a_i - A^2p_i + nA^2\tag{by \prettyref{lem:pij}}\\ 
	 &\leq (n+1)A^2\tag{by \prettyref{eq:p_p}}
	\end{align*}
The case $i \in \cN$ is completely symmetric. This concludes the proof of \prettyref{thm:GCD}.
\end{proof}

Notice that it is straightforward to improve the bounds of \prettyref{thm:GCD} further (\eg getting rid of the factor $n+1$). However, since there is no need for that in order to perform the matrix reduction, we do not do this additional effort. Also we could not find a \TC circuit which yields the bound $x_i \leq \frac{1}{2}A$ (which is achievable in \L by \cite{MH94}). 

\section{Matrix reduction and subgroup membership problem }\label{sec:MatrixReduction}\label{se:matrix_reduction}\label{se:presentation}
In \cite{MacDonaldMNV15}, the so-called matrix reduction procedure converts an arbitrary matrix of coordinates into its full form and, thus, is an essential step for solving the subgroup membership problem and several other problems.
It was first described in \cite{Sim94}~-- however, without a precise complexity estimate. In this section, we repeat the presentation from \cite{MacDonaldMNV15} and show that for fixed $c$ and $r$, it can be actually computed  uniformly for groups in \Ncr in \TC~-- in the case that the inputs are given in unary (as words). If the inputs are represented as words with binary exponents, then we still can show that it is \TC-Turing-reducible to \ExtGCD.
In Section~\ref{sec:subgroup}, we defined the matrix representation of subgroups of nilpotent groups. We adopt all notation from Section~\ref{sec:subgroup}.

As before, let $c,r \in \N$ be fixed and let $(a_1, \dots, a_m)$ be the standard Mal'cev basis of $F_{c,r}$. Let $G \in \Ncr$ be given as \qmp, \ie as a matrix in full form (either with unary or binary coefficients).
We define the following operations on tuples $(h_{1},\ldots,h_{n})$ (our subgroup generators) of elements of $G$ and the corresponding operations on the associated matrix, with the goal of converting $(h_{1},\ldots,h_{n})$ to a sequence in full form generating the same subgroup $H=\gen{ h_{1}, \ldots, h_{n}}$:
	\begin{enumerate}[(1)]
		\item Swap $h_{i}$ with $h_{j}$.  This corresponds to swapping row $i$ with row $j$.\label{row:swap}
		\item Replace $h_i$ by $h_ih_j^l$ ($i\neq j,\; l\in\Z$). This corresponds to replacing row $i$ by $\coords{h_ih_j^l}$.\label{row:substraction}
		\item Add or remove a trivial element from the tuple. This corresponds to adding or removing a row of zeros; or (3') a row of the form $(0\;\ldots\; 0\; e_i\; \alpha_{i+1}\;\ldots\; \alpha_m)$, where $i\in \mathcal T$ and $a_i^{-e_i}=a_{i+1}^{\alpha_{i+1}}\cdots a_m^{\alpha_m}$.\label{row:trivial}
		\item Replace $h_i$ with $h_i^{-1}$. This corresponds to replacing row $i$ by $\coords{h_i^{-1}}$. \label{row:inverse}
		\item Append an arbitrary product $h_{i_1}^{l_1}\cdots h_{i_k}^{l_k}$ with $i_1,\dots,i_k \in \oneset{1, \dots, n}$ and $l_1,\dots,l_k \in \Z$ to the tuple: add a new row with $\coords{h_{i_1}^{l_1}\cdots h_{i_k}^{l_k}}$.\label{row:add_linear}
	\end{enumerate}%
Clearly, all these operations preserve $H$. 
\begin{lemma}\label{lem:operationsTC}
	On input of a \qmp of $G \in \cN_{c,r}$ in unary (resp.\ binary) and a matrix of coordinates $A$ given in unary (resp.\ binary), operations (1)--(5) can be done in \TC.  The output matrix will be also encoded in unary (resp.\ binary). For operations (2) and (5), we require that the exponents $l$, $l_1, \dots, l_k$ are given in unary (resp.\ binary).
	
	Moreover, as long as the rows in the matrix which are changed are pairwise distinct, a polynomial number of such steps can be done in parallel in \TC.
\end{lemma}
\begin{proof}
	Operations (1) and (3), clearly can be done in \TC. Notice that operation (3') means simply that a row of the \qmp of $G$ is appended to the matrix.
	
	In the unary case, it follows directly from  Theorem~\ref{thm:nilpotentExpWP} that operations (2), (4), and (5) are in \TC because, since $l$, $l_1, \dots, l_k$ are given in unary, the respective group elements can be written down as words. 
	
	In the case of binary inputs, (5) works as follows ((2) and (4) analogously): by Lemma~\ref{lem:malcev} (ii), there are functions
	$q_{1},\ldots,q_{m}\in\Z[x_1, \dots, x_m,z]$ such that for every $h\in F_{c,r}$ with $\coords{h}=(\gamma_{1},\ldots,\gamma_{m})$ anda $l\in\Z$, we have
	$\coordsj{h^{l}}{i} = q_{i}(\gamma_{1},\ldots,\gamma_{m}, l)$
	in  $F_{c,r}$.
	These functions can be used to compute 
	$\coords{h_{i_j}^{l_j}}$ for $j=1, \dots, k$. After that, $h_{i_1}^{l_1}\cdots h_{i_k}^{l_k}$ can be written down as word with binary exponents and Theorem~\ref{thm:nilpotentExpWP} can be applied.
\end{proof}
Using the row operations defined above, in \cite{MacDonaldMNV15} it is shown how to reduce any coordinate matrix to its unique full form. Let us repeat these steps:

Let $A_{0}$ be a matrix of coordinates, as in (\ref{Eqn:CoordinateMatrix}) in Section~\ref{sec:subgroup}. Recall that $\pi_k$ denotes the column index of the $k$-th pivot (of the full form of $A_0$). We produce matrices $A_{1}, \ldots, A_{s}$, where $s$ is the number of pivots in the full form of $A_{0}$, such that for every $k=1,\ldots, s$ the first $\pi_k$ columns of $A_{k}$ form a matrix satisfying conditions \ref{li:std_echelon}-\ref{li:std_torsion} of being a full sequence, condition \ref{li:std_full} is satisfied for all $i<\pi_{k+1}$, and 
$A_{s}$ is the full form of $A_{0}$.  Here we formally denote $\pivot{s+1}=m+1$. 
Set $\pivot{0}=0$ and assume that $A_{k-1}$ has been constructed for some $k\geq 1$. In the steps below we construct $A_{k}$.
We let $n$ and $m$ denote the number of rows and columns, respectively, of $A_{k-1}$. At all times during the computation, $h_{i}$ denotes 
the group element corresponding to row $i$ of $A_{k}$ and $\alpha_{ij}$ denotes 
the $(i,j)$-entry of $A_{k}$, which is $\coordsj{h_{i}}{j}$.  These may change 
after every operation.
\begin{description}
	\item [\sc Step 1.]
	Locate the column  $\pivot{k}$ of the next pivot, which is the minimum integer $\pivot{k-1}<\pivot{k}\leq m$ such that 
	$\alpha_{i \pivot{k}}\neq 0$ for at least one $k \leq i \leq n$. If no such integer exists, then $k-1=s$ and 
	$A_{s}$ is already constructed.  Otherwise, set $A_{k}$ to be a copy of $A_{k-1}$ and denote $\pivot{}=\pivot{k}$. Compute a linear expression of \[d = {\rm gcd}(\alpha_{k\pivot{}}, \ldots, \alpha_{n\pivot{}})= l_{k}\alpha_{k \pivot{}} + \cdots +  l_{n}\alpha_{n\pivot{}}.\]
	Let $h_{n+1} = h_{k}^{l_{k}} \cdots h_{n}^{l_{n}}$ and note that $h_{n+1}$ has coordinates of the form
	\[\coords{h_{n+1}} = (0,\ldots,0,d, \ldots)\] with $d$ occurring in position $\pivot{}$.
	Perform operation~\ref{row:add_linear} to append $h_{n+1}$ as row $n+1$ of $A_{k}$.
	\item [\sc Step 2.] For each $i=k, \ldots, n$, perform operation~\ref{row:substraction} to replace row $i$ by 
	$\coords{h_i\cdot h_{n+1}^{-\alpha_{i\pivot{}}/d}}.$ and for each $i=1,\ldots,k-1$, use \ref{row:substraction} to replace row $i$ by 
	$\coords{h_{i} \cdot h_{n+1}^{-\lfloor \alpha_{i\pivot{}}/d\rfloor}}$.
	After that, swap row $k$ with row $n+1$ using \ref{row:swap}.  At this point, properties 
	\ref{li:std_echelon}-\ref{li:std_reduced} hold on the first $k$ columns of $A_{k}$.
	\item [\sc Step 3.] If $\pivot{}\in\mathcal{T}$, we additionally ensure condition~\ref{li:std_torsion} as follows. Perform row operation~(3'), with respect to $\pivot{}$, to append a trivial element 
	$h_{n+2}$ with $\coords{h_{n+2}}=(0,\ldots,0,e_{\pivot{}},\ldots)$ to $A_{k}$. 
	Let $\delta=\gcd(d,e_{\pivot{}})$ and compute the linear expression $\delta=n_{1}d+n_{2}e_{\pivot{}}$, with 
	$|n_{1}|,|n_{2}|\leq\max\{d,e_{\pivot{}}\}$. Let 
	$h_{n+3}= h_{k}^{n_{1}} h_{n+2}^{n_{2}}$ and append this row to $A_{k}$, as row $n+3$. Note that 
	$\coords{h_{n+3}}=(0,\ldots,0,\delta,\ldots)$, with $\delta$ in position $\pivot{}$.  
	Replace row $k$ by $\coords{h_{k}\cdot h_{n+3}^{-d/\delta}}$ and row $n+2$ by 
	$\coords{h_{n+2}\cdot h_{n+3}^{-e_{\pivot{}}/\delta}}$, 
	producing zeros in column $\pivot{}$ in these rows. Swap row $k$ with row $n+3$.  
	At this point, 
	\ref{li:std_echelon}, \ref{li:std_positive}, and \ref{li:std_torsion} hold 
	(for the first $\pivot{k}$ columns) but \ref{li:std_reduced} 
	need not, since the pivot entry is now $\delta$ instead of $d$. 
	For each $j=1,\ldots,k-1$, replace row $j$ by $\coords{h_{j}\cdot h_{k}^{-\lfloor \alpha_{j\pivot{}}/\delta\rfloor}}$, 
	ensuring \ref{li:std_reduced}.
	\item [\sc Step 4.] Identify the next pivot $\pivot{k+1}$ (like in Step 1). If 
	$\pivot{k}$ is the last pivot, we set $\pivot{k+1}=m+1$.
	We now ensure condition \ref{li:std_full} for $i< \pi_{k+1}$. 
	Observe that Steps 1-3 preserve $\genr{ h_{j} }{\pivot{j}\geq i}$ for all 
	$i<\pivot{k}$. Hence \ref{li:std_full} holds in $A_{k}$ for $i<\pivot{k}$ since it 
	holds in $A_{k-1}$ for the same range.  Now consider $i$ in the range 
	$\pivot{k}\le i<\pivot{k+1}$. It suffices to establish (vi.i) for all $j>k$ and (vi.ii) for 
	$\pivot{k}$ only.
	To obtain (vi.i), notice that $h_{k}^{-1}h_{j}h_{k},h_{k}h_{j}h_{k}^{-1}\in \genr{ h_\ell }{ \ell>k}$ if, and only if, 
	$[h_{j},h_{k}^{\pm 1}]\in \genr{ h_{\ell}}{ \ell>k}$. 
	Further, note that the subgroup generated by
	\[S_{j}=\{1,h_j, [h_j,h_k],\ldots, [h_j,h_k,\ldots,h_k]\},\]
	where $h_k$ appears $m-\pivot{k}$ times in the last commutator, is closed under commutation with $h_k$ since if $h_{k}$ appears more than $m-\pivot{k}$ times then 
	the commutator is trivial. An inductive argument shows that the subgroup $\gen{ S_{j}}$ coincides with $\langle h_k^{-\ell} h_j h_k^\ell\mid 0\le \ell\le m-\pi_k\rangle$. Similar observations can be made for conjugation by $h_k^{-1}$. Therefore, appending via operation \ref{row:add_linear} rows $\coords{h_k^{-\ell} h_j h_k^\ell}$ for all 
	$1\leq |\ell| \leq m-\pivot{k}$ and all $k< j\le n+3$ delivers (vi.i) for all $j>k$. 
	Note that (vi.i) remains true for $i<\pivot{k}$. 
	
	To obtain (vi.ii), in the case  $\pi_{k}\in\mathcal{T}$, we add row $\coords{h_{k}^{e_{k}/\alpha_{k\pi_{k}}}}$. Note that this element commutes with $h_k$ and therefore (vi.i) is preserved.
	\item [\sc Step 5.] Using operation \ref{row:trivial}, eliminate all zero rows.  The matrix $A_{k}$ is now constructed.\smallskip
\end{description}

We have to show that each step can be performed in \TC  given that all Mal'cev coordinates are encoded in unary (resp.\ in $\TC(\ExtGCD)$ if Mal'cev coordinates are encoded in binary). Since the total number of steps is constant (only depending on the nilpotency class and number of generators), this gives a \TC (resp.\ $\TC(\ExtGCD)$) circuit for computing the full form of a given subgroup.
\begin{description}
	\item [\sc Step 1.] The next pivot can be found in \TC since it is simply the next column in the matrix with a non-zero entry, which can be found as a simple Boolean combination of test whether the entries are zero. In the unary case, by \prettyref{thm:GCD}, $d = {\rm gcd}(\alpha_{k\pivot{}}, \ldots, \alpha_{n\pivot{}})$ can computed in \TC together with $l_k, \dots,  l_n$ encoded in unary such that
	$d = l_{k}\alpha_{k \pivot{}} + \cdots +  l_{n}\alpha_{n\pivot{}}$. Now, by Lemma~\ref{lem:operationsTC}, Step 1 can be done in \TC.
	
	In the binary case, $d$ and $l_k, \dots,  l_n$ can be computed using \ExtGCD. Hence, by Lemma~\ref{lem:operationsTC}, Step 1 can be done in $\TC(\ExtGCD)$.
	\item [\sc Step 2.] The numbers $\lfloor \alpha_{i\pivot{}}/d\rfloor$ (either in unary or binary) can be computed in \TC for all $i$ in parallel by Theorem~\ref{thm:divisionTC}. After that one operation (2) is applied to each row of the matrix. By Lemma~\ref{lem:operationsTC}, this can be done in parallel for all rows in \TC. Finally, swapping rows $k$ and $n+1$ can be done in \TC.
	\item [\sc Step 3.] As explained in Section~\ref{sec:WP},
	$\cT$ and $e_i $ for $i\in \cT$ can be read directly from the \qmp. Thus, it can be decided in \TC whether Step 3 has to be executed. 
	Appending a new row is in \TC. Computing $\gcd(d, e_\pi) = d = n_1 d  n_2 e_\pi$ is in \TC by \prettyref{ex:gcdeasy} (in the unary case) and in $\TC(\ExtGCD)$ in the binary case. After that one operation (5) is followed by two operations (2), one operation (1), and, finally, $k-1$ times operation (2), which all can be done in \TC again by Lemma~\ref{lem:operationsTC}.
	\item [\sc Step 4.] The next pivot can be found in \TC as outlined in Step 1.
	After that, Step 4 consists of an application of a constant number (only depending on the nilpotency class and number of generators) of operations (5) and thus, by Lemma~\ref{lem:operationsTC}, is in \TC.	
	\item [\sc Step 5.] Clearly that is in \TC.
\end{description}
Thus, we have completed the proof of our main result:

\begin{theorem}
	\label{thm:Compute_std_form}
	Let $c,r \in \N$ be fixed. The following problem is in \TC: given a unary encoded \qmp of $G \in \cN_{c,r}$
	and $h_{1},\ldots,h_{n}\in G$, compute the full form 
	of the associated matrix of coordinates encoded in unary and hence the unique full-form sequence $(g_{1},\ldots,g_{s})$ generating $\langle h_{1},\ldots,h_{n}\rangle$. 
	Moreover, if the $G$ and $h_{1},\ldots,h_{n}$ are given in binary, then  the full-form sequence with binary coefficients can be computed in $\TC(\ExtGCD)$.
\end{theorem}

\subsection{Subgroup membership problem}\label{se:membership_problem}
\label{se:presentation}

We  can  now  apply  the  matrix  reduction  algorithm  to  solve  the subgroup membership problem  in \TC.

\begin{theorem}\label{thm:logspace_membership}
	Let $c,r \in \N$ be fixed. The following problem is in \TC (resp.\ $\TC(\ExtGCD)$ for binary inputs): given a \qmp of $G \in \cN_{c,r}$, elements $h_{1},\ldots,h_{n}\in G$ and $h\in G$, decide whether or not $h$ is an element of the subgroup $H=\gen{ h_{1},\ldots, h_{n}}$. 

Moreover, if $h\in H$, the circuit computes the unique expression $h=g_{1}^{\gamma_{1}}\cdots g_{s}^{\gamma_{s}}$ where $(g_{1},\ldots,g_{s})$ is the full-form sequence for $H$ 
with the $\gamma_i$ encoded in unary (resp.\ binary).

Alternatively, for unary inputs, the output can be given as word $h=h_{i_1}^{\epsilon_{1}}\cdots h_{i_t}^{\epsilon_{t}}$ where $i_j \in \{1, \ldots, n\}$ and $\epsilon_j = \pm 1$.
\end{theorem}
Note that we do not know whether there is an analog of the second type of output for binary inputs. A possible way of expressing the output would be as a word with binary exponents over $h_{1},\ldots,h_{n}$. However, simply applying the same procedure as for unary inputs will not lead to a word with binary exponents.

\begin{proof}
	The circuit works as follows: first, the
	the full form $A$ of the coordinate matrix corresponding to 
	$H$ and the standard-form sequence $(g_1, \ldots, g_s)$ are computed in \TC (resp.\ $\TC(\ExtGCD)$) using \prettyref{thm:Compute_std_form}. 
	As before, denote by $\alpha_{ij}$ the $(i,j)$-entry of $A$ and by 
	$\pivot{1}, \ldots, \pivot{s}$ its pivots.
	
	By Lemma~\ref{lem:UniqueStandardForm}, any element of $H$ can be written as $g_1^{\gamma_1} \cdots g_s^{\gamma_s}$. We show how to find these exponents. Denote $h^{(1)} = h$ and $\coords{h^{(j)}}=(\beta_1^{(j)}, \ldots, \beta_m^{(j)})$, with $h^{(j)}$ being defined below. For $j=1, \ldots, s$, do the following. If $\beta_{l}^{(j)} \neq 0$ for 
	any  $1\leq l < \pi_j$, then $h\notin H$. Otherwise, 
	check whether $\alpha_{j{\pi_j}}$ divides $\beta_{\pi_{j}}^{(j)}$. If not, then $h\notin H$. If yes, let 
	\[
	\gamma_{j} = \beta_{\pi_{j}}^{(j)}/{\alpha_{j\pi_j}} \quad \text{and} \quad h^{(j+1)} = g_j^{-\gamma_j}h^{(j)}. 
	\]
	If $j<s$, continue to $j+1$.  If $j=s$, then $h=g_{1}^{\gamma_{1}}\cdots g_{s}^{\gamma_{s}}\in H$ if $h^{(s+1)}=1$ and 
	$h\notin H$ otherwise.
	
	Since $s$ is bounded by a constant, there are only a constant number of steps. Each step can be done in \TC by \prettyref{thm:divisionTC} (division) and \prettyref{thm:nilpotentExpWP} (computation of Mal'cev coordinates).
	
	For the second type of output in the unary case, while performing the matrix reduction, we store for every row of the matrix also how that row can be expressed as a word over the subgroup generators $h_1, \dots, h_n$ (here, we need the unary inputs, as otherwise the group elements cannot be expressed as words in polynomial space). In every operation on the matrix these words are updated correspondingly, which clearly can be done in \TC.  
	In the end after writing $h=g_{1}^{\gamma_{1}}\cdots g_{s}^{\gamma_{s}}$, every $g_i$ can be substituted by the respective word. 
\end{proof}

Since abelian groups are nilpotent, we obtain:

\begin{corollary}[]\label{cor:smpAb}
Let $r$ be fixed. The following problem is in \TC:
Given a list $h_1, \dots, h_n \in \Z^r$ and $g \in \Z^r$ (all as words over the generators), decide whether $g \in \gen{h_1, \dots, h_n}$. Moreover, in the case of a positive answer, compute $x_1, \dots, x_n \in \Z$ in unary such that $g = x_1h_1 + \dots + x_nh_n$.

In other words: for fixed $r$, given a unary encoded system of linear equations $(A,b)$ with $A \in \Z^{r \times n}$ and $b \in \Z^r$, a unary encoded solution $x \in \Z^n$ with $Ax=b$ can be computed in \TC. 
\end{corollary}

\subsection{Subgroup presentations}\label{se:presentation}

The full-form sequence associated to a subgroup $H$ forms a Mal'cev basis for $H$. This allows us to compute a consistent
nilpotent presentation for $H$. Note, however, that the resulting presentation is \emph{not} a \qmp (although it can be transformed into one, see \prettyref{prop:find_presentation_poly})~-- partly this is due to the fact that, in general, $H\notin \Ncr$.
The following is the extended version of \cite[Thm.\ 3.11]{MacDonaldMNV15}:
\begin{theorem}\label{thm:EffectiveCoherence}
Let $c,r \in \N$ be fixed. The following is in \TC for unary inputs and in $\TC(\ExtGCD)$ for binary inputs:

\emph{Input:} a \qmp for $G \in \cN_{c,r}$ and elements $h_{1},\ldots,h_{n}\allowbreak\in G$.

\emph{Output:} a consistent nilpotent presentation for $H = \gen{ h_{1},\ldots,h_{n}}$ given by a list of generators $(g_{1},\ldots,g_{s})$ and numbers  $\mu_{ij},\alpha_{ijk},\beta_{ijk}\in\Z$ encoded in unary (resp.\ binary) for $1 \leq i < j < k \leq s$ representing the relations (\ref{stdpolycyclic1})-(\ref{stdpolycyclic3}).
\end{theorem}

\begin{proof}
	First, the full sequence $(g_{1},\ldots,g_{s})$ 
	for $H$ is computed in \TC (resp.\ $\TC(\ExtGCD)$) according to \prettyref{thm:Compute_std_form}. Let 
	$H_{i}=\langle g_{i},g_{i+1},\ldots,g_{s}\rangle$. In the proof of \cite[Thm.\ 3.11]{MacDonaldMNV15}, it is shown that
	$(g_{1},\ldots,g_{s})$ is a Mal'cev basis 
	for $H$. Hence, it remains to compute the relators 
	(\ref{stdpolycyclic1})-(\ref{stdpolycyclic3}) in order to give a consistent 
	nilpotent presentation of $H$. 
	The order $e_{i}'$ of $g_{i}$ modulo $H_{i+1}$ is simply 
	$e_{i}/\coordsj{g_{i}}{\pivot{i}}$ (as before $\cT$ and $e_i $ for $i\in \cT$ can be read from the \qmp). 
	Each relation (\ref{stdpolycyclic1}) can be computed using the \TC (resp.\ $\TC(\ExtGCD)$) circuit of \prettyref{thm:logspace_membership} 
	with input $g_{i}^{e_{i}'}$ and $H_{i+1}=\langle g_{i+1},\ldots, g_{s}\rangle$.  
	Since $g_{i}^{e_{i}'}\in H_{i+1}$ and $(g_{i+1},\ldots,g_{s})$ is the unique full sequence for $H_{i+1}$, the membership algorithm returns the expression 
	on the right side of (\ref{stdpolycyclic1}). Relations (\ref{stdpolycyclic2}) and 
	(\ref{stdpolycyclic3}) are established using the same method. Note that there are only a constant number of relations to establish~-- so everything can be done in \TC (resp.\ $\TC(\ExtGCD)$).
\end{proof}

\section{More algorithmic problems}\label{sec:more}
\subsection{Homorphisms and kernels}

Given nilpotent groups $G$ and $H$ and a subgroup $K\leq G$ and a generating set $g_{1},\ldots,g_{n}$ of $K$, a homomorphism 
$\phi:K \to H$ can be specified by a list of elements $h_{1},\ldots,h_{n}$ where $\phi(g_{i})=h_{i}$ for 
$i=1,\ldots,n$. For a homomorphism, we consider the problem of finding a generating set for its kernel, and given $h\in \phi(K)$ finding $g\in G$ 
such that $\phi(g)=h$. Following \cite{MacDonaldMNV15}, both problems are solved using matrix reduction 
in the group $H\times G$.

\begin{theorem}[Kernels and preimages]\label{thm:KernelAndPreimage}
Let $c,r \in \N$ be fixed. The following is in \TC for unary inputs and in $\TC(\ExtGCD)$ for binary inputs: On input of
	\begin{itemize}
	    \item $G, H \in \Ncr$ given as \qmps,
		\item a subgroup $K=\langle g_{1},\ldots,g_{n}\rangle\leq G$,
		\item a list of elements $h_{1},\ldots,h_{n}$ defining a homomorphism $\phi:K\rightarrow H$ via $\phi(g_{i})=h_{i}$, and 
		\item optionally, an element $h\in H$ guaranteed to be in the image of $\phi$, 
	\end{itemize}
	compute 
	\begin{enumerate}
		\item a generating set $X$ for the kernel of $\phi$, and \label{kernel}
		\item an element $g\in G$ such that $\phi(g)=h$.\label{preimage}
	\end{enumerate}
	 In case of unary inputs, $X$ and $g$ will be returned as words, and for binary inputs, as words with binary exponents.
\end{theorem}

\begin{proof}
	Let $(a_1, \dots, a_m)$ be the standard Mal'cev basis of $F_{c,r}$ and $(b_1, \dots, b_{m'})$ the standard Mal'cev basis of $F_{c,2r}$
	We have two embeddings of $\phi_H, \phi_G : F_{c,r} \to F_{c,2r}$ with $\phi_H(a_i) = b_i$  and $\phi_G(a_i) = b_{r + i}$ for $i = i, \dots, r$. We can assume that the Mal'cev basis of $F_{c,2r}$ is chosen in such a way that these embeddings send all Mal'cev generators of $F_{c,r}$ to Mal'cev generators of $F_{c,2r}$. Note that we have $\phi_H(F_{c,r}) \cap \phi_G(F_{c,r}) = \oneset{1}$.
	
	Thus, we can read all relators of $H$ and $G$ in $F_{c,2r}$ via the embeddings $\phi_H$ and $\phi_G$, respectively. To obtain a \qmp of $H \times G$, we simply need to add the relations that $H$ and $G$ commute~-- that is we need to introduce additional relations $b_i = 1$ for all Mal'cev generators which are not in the image of  $\phi_G$ or $\phi_H$. As the new \qmp is basically a copy of those of $H$ and $G$, it can be computed in \TC. 
	From now on we work only in the direct product $H \times G \in \cN_{c,2r}$ and identify $G$ and $H$ with their images under $\phi_G$ and $\phi_H$.

	Let $Q=\langle h_{i} g_{i}\, |\, 1\leq i\leq n\rangle$  and let $W=(v_{1} u_{1}, \ldots, v_{s} u_{s})$ be the sequence in full form 
	for the subgroup $Q$, 
	where $u_{i}\in G$ and $v_{i}\in H$.  Let  $0\leq r\leq s$ 
	be the greatest integer such that $v_{r}\neq 1$ (with $r=0$ if all $v_{i}$ are 1).
	Set $X=(u_{r+1},\ldots,u_{n})$ and $Y=(v_{1},\ldots,v_{r})$.
	In \cite[Thm.\ 4.1]{MacDonaldMNV15} it is shown that $X$ is the full-form sequence for the kernel of $\phi$ and $Y$ is the full-form sequence for the image.

	Now, to solve \ref{kernel}, it suffices to compute $W$ using \prettyref{thm:Compute_std_form} and return the corresponding $X$ as defined above. 
	For
	\ref{preimage}, apply \prettyref{thm:logspace_membership} to express $h$ as $h=v_{1}^{\beta_{1}}\cdots v_{r}^{\beta_{r}}$~-- then 
	return $g=u_{1}^{\beta_{1}}\cdots u_{r}^{\beta_{r}}$. 
\end{proof}

\subsection{Centralizers}
Before we focus on the conjugacy problem, we need one more preliminary result: the problem of computing centralizers.

\begin{theorem}[Centralizers]\label{thm:Centralizer}
	Let $c,r \in \N$ be fixed. The following is in \TC for unary inputs and in $\TC(\ExtGCD)$ for binary inputs:
	
	On input of some $G \in \Ncr$ given as \qmp and an element $g\in G$, compute
	a generating set $X$ for the centralizer of $g$ in $G$ (in case of binary inputs, the generating set will be given as set of words with binary exponents).
\end{theorem}
\begin{proof}
	Let $F_{c,r}=\Gamma_{0}\geq \Gamma_{1}\geq\cdots\geq\Gamma_{c+1}=1$ be the lower central 
	series of $F_{c,r}$. Clearly this central series projects onto a central series of $G$ and we simply write $\Gamma_{i}$ for its projection in $G$. Denote with $A=(a_1, \dots a_m)$ the standard Mal'cev basis of $F_{c,r}$, which is associated to the lower central series~-- in particular $a_1, \dots, a_r$ is a generating set for $F_{c,r}$.
	
	We proceed by induction on $c$. If $c=1$, then $F_{c,r}$ and $G$ are abelian and $C(g)=G$ so the output is $\oneset{a_1, \dots, a_r}$.  Assume 
	that the theorem holds for groups in $\cN_{c-1,r}$~-- 
	in particular, for $G/\Gamma_{c}$ (we obtain a \qmp of $G/\Gamma_{c}$ by simply forgetting about the Mal'cev generators in $\Gamma_c$). 
	A generating set $K=\{k_{1}\Gamma_{c},\ldots,k_{n}\Gamma_{c}\}$ for the centralizer of $g\Gamma_{c}$ in $G/\Gamma_{c}$ can be computed in \TC (resp.\ $\TC(\ExtGCD)$) by induction. 
	Let \[J=\gen{ k_{1},\ldots,k_{n}, a_{m'},\ldots,a_{m}},\] where $\oneset{a_{m'}, \dots, a_m} = A \cap \Gamma_c$.  
	Then $J$ is the preimage of 
	$\langle K\rangle$ under the homomorphism 
	$G\rightarrow G/\Gamma_{c}$.  
	Define $f:J\rightarrow G$ by 
	\[
	f(u) = [g,u].
	\]
	Since $u\in J$, $u$ commutes with $g$ modulo $\Gamma_{c}$; hence, $[g,u]\in \Gamma_{c}$ and so $\mathrm{Im}(f)\subseteq \Gamma_{c}$. Moreover, $f$ is a 
	homomorphism: we have
	\[
	f(g,u_{1}u_{2})=[g,u_{1}u_{2}]= 
	 [g,u_{2}][g,u_{1}][[g,u_{1}],u_{2}],
	\]  
	and $[g,u_{1}]\in \Gamma_{c}$; therefore, $[[g,u_{1}],u_{2}]\in \Gamma_{c+1}=1$, and 
	$[g,u_{1}]$ and $[g,u_{2}]$ commute, both being elements of the abelian group 
	$\Gamma_{c}$. 
	
	If $h$ commutes with $g$, then $h\Gamma_{c}\in \langle K\rangle$, \ie\ $h\in J$. Thus, the centralizer of $g$ is precisely the kernel of $f:J\rightarrow \Gamma_{c}$. A generating set can be computed in \TC (resp.\ $\TC(\ExtGCD)$) using 
	\prettyref{thm:KernelAndPreimage}.
\end{proof}

\subsection{The conjugacy problem}
Now, we can combine the previous theorems to solve the conjugacy problem in \TC following \cite[Thm.\ 4.6]{MacDonaldMNV15}.

\begin{theorem}[Conjugacy Problem]\label{thm:CP}
Let $c,r \in \N$ be fixed. The following is in \TC for unary inputs and in $\TC(\ExtGCD)$ for binary inputs:
On input of some $G \in \Ncr$ given as \qmp and elements $g,h\in G$, either 
\begin{itemize}
\item produce some $u\in G$ such that $g=h^u$, or 
\item determine that no such element $u$ exists.  
\end{itemize}
	 In case of unary inputs, $u$ will be returned as a word, for binary inputs, as a word with binary exponents.
\end{theorem}

\begin{proof}
	 Again we proceed by induction on $c$.
	If $c=1$, then $G$ is abelian and $g$ is conjugate to $h$ if and only if 
	$g=h$. If so, we return $u=1$.
	
	Now let us assume $c>1$ and that the theorem holds for any nilpotent group of class 
	$c-1$~-- in particular, for $G/\Gamma_{c}$. We use the notation as in the proof of \prettyref{thm:Centralizer}. 
	
	The first step of the circuit is to check conjugacy of $g \Gamma_{c}$ and 
	$h \Gamma_{c}$ in $G/\Gamma_{c}$ which can be done in \TC by induction. 
	If these elements are not conjugate, then 
	$g$ and $h$ are not conjugate and the overall answer is `No'. Otherwise, we obtain some $v \Gamma_{c}\in G/\Gamma_{c}$
	such that $g \Gamma_{c} = h^v  \Gamma_{c}$. 
	
	Let $\phi:G\rightarrow G/\Gamma_{c}$ be the canonical homomorphism, 
	$J=\phi^{-1}(C(g\Gamma_{c}))$ (where $C(g\Gamma_{c})$ denotes the centralizer of $g\Gamma_{c}$), and define $f:J\rightarrow \Gamma_{c}$ by $f(x)=[g,x]$.  As in the proof of \prettyref{thm:Centralizer}, 
	the image of $f$ is indeed 
	in $\Gamma_{c}$ and $f$ is a homomorphism. We claim that 
	$g$ and $h$ are conjugate if and only if $g^{-1}h^{v}\in f(J)$.  Indeed, if there exists $w\in G$ such that $g=h^{vw}$, then 
	\[
	1 \cdot \Gamma_{c} = g^{-1} w^{-1} h^{v} w \cdot \Gamma_{c} = [g,w] \cdot \Gamma_{c},
	\]
	hence $w\in J$, so $w^{-1}\in J$ as well.  Then $g^{-1}h^{v}=[g,w^{-1}] \in f(J)$, as required.  The converse is immediate. 
	So it suffices to express, if possible, $g^{-1}h^{v}$ as $[g,w]$ with $w\in J$, in which case the conjugator is $u=vw^{-1}$.
	
	Now, the circuit computes a generating set $\{w_{1} \Gamma_{c}, \ldots, w_{n} \Gamma_{c}\}$ for 
	$C(g \Gamma_{c})$ using 
	\prettyref{thm:Centralizer}.  
	Then $J$ is generated by
	$\{w_{1},\ldots,w_{n},a_{m'}, \dots, a_m\}$, where again $\oneset{a_{m'}, \dots, a_m} = A \cap \Gamma_c$.
	After that, $\coords{g^{-1} h^v }$ is computed and
	\prettyref{thm:logspace_membership} used to determine whether $g^{-1}h^v\in f(J)$. If so, \prettyref{thm:KernelAndPreimage} is applied to find some $w\in G$ such that 
	$g^{-1} h^v=f(w)$. Finally, $u=vw^{-1}$ is returned in case all previous tests succeed. 
	Since we only concatenate a fixed constant number of \TC (resp.\ $\TC(\ExtGCD)$) computations, the whole computation is in \TC (resp.\ $\TC(\ExtGCD)$) again.
\end{proof}

\begin{remark}
We want to outline briefly how in the unary case the bounds of \cite[Thm.\ 4.6]{MacDonaldMNV15} can be used to directly solve the conjugacy problem of nilpotent groups in \TC. Since \cite[Thm.\ 4.6]{MacDonaldMNV15} is for a non-uniform setting, we fix a nilpotent group $G$ with generating set $A$. 
	Let $g,h$ be words over $A^{\pm1}$ as inputs for the conjugacy problem with of total length $n$.	
	By \cite[Thm.\ 4.6]{MacDonaldMNV15}, the length of conjugators is polynomial in $n$. By using binary exponents, the conjugators can be written with respect to a Mal'cev basis of $G$ using only $C\log n$ bits for some constant $C$ which only depends on $G$ (this is a well-known fact~-- see \eg  \cite[Thm.\ 2.3]{MacDonaldMNV15}). In particular, for all possible conjugators $u$ which have bit-length at most $C\log n$, it can be checked in parallel by a uniform family of \TC circuits whether $g= h^u$ in $G$ by using the circuits for the word problem \cite{Robinson93phd} (note that for this purpose each $u$ can be written down in unary since it is of length at most $n^C$).
\end{remark}

\section{Computing \qmps}
\label{sec:Uniform2}

The results in the previous sections always required that the group is given as a \qmp. However, we can use
\prettyref{thm:Compute_std_form} to transform an arbitrary presentation with at most $r$ generators of a group in \Ncr into a \qmp.

\begin{proposition}\label{prop:find_presentation_poly}
Let $c$ and $r$ be fixed integers. The following is in \TC: given an arbitrary finite presentation 
with generators $a_1, \dots, a_r$ of a group $G \in \Ncr$ (as a list of relators given as words over $\oneset{a_1, \dots, a_r}^{\pm1}$), 
compute a \qmp of $G$ (encoded in unary) and an explicit isomorphism.

Moreover, if the relators are given as words with binary exponents, then the binary encoded \qmp can be computed in $\TC(\ExtGCD)$.
\end{proposition}

\begin{proof}
	Let $A=\oneset{a_1, \dots, a_r}$ and let $R$ be the set of relators, \ie $G$ is presented as $G=\genr{A}{R}$. Let $F = F_{c,r} = \gen{a_1, \dots, a_r}$ be the free nilpotent group of class $c$ on generators $A$. Let $B= \oneset{b_1, \dots, b_m}$ be the standard Mal'cev basis of $F$ such that $b_i = a_i$ for $i = 1, \dots, r$ and let $S$ denote the set of relations such that $\genr{B}{S}$ is a consistent nilpotent presentation for $F$.

	Consider the natural surjection $\phi:F\rightarrow G$ 
	and let $N=\ker(\phi)$, which is the normal closure of $R$ in $F$.  
	Denoting $R=\{r_1,\ldots, r_k\}$, $N$ is generated  by iterated commutators $[\ldots[[r_i,x_1],x_2],\ldots,x_j]$, where $i=1,\ldots,k$, $j\le c$, and $x_1,\ldots, x_j\in A\cup A^{-1}$.
	The total length of these generators is linear since $c$ and $r$ are constant. Using \prettyref{thm:Compute_std_form} in the group $F$, we can produce the full-form sequence $T$ for $N$ in \TC (resp.\ in $\TC(\ExtGCD)$ for binary inputs).
	
	Now $G\simeq \langle B\mid S\cup T\rangle$ and by \prettyref{lem:quotient_presentation} this is a (consistent) \qmp.
\end{proof}
\begin{remark}
Because of \prettyref{prop:find_presentation_poly}, in all theorems above where the input is a \qmp, we can also take an arbitrary $r$-generated presentation of a group in $\Ncr$ as input. However, be aware that for the word problem (\prettyref{thm:nilpotentExpWP} and \prettyref{cor:nilpotentExpWP}) the complexity changes from \TC to $\TC(\ExtGCD)$ in the binary case.
\end{remark}

\section{Power problem and conjugacy in wreath products of nilpotent groups}\label{sec:pp}

In \cite{MiasnikovVW17}, the conjugacy problem in iterated wreath products of abelian is shown to be in \TC (for a definition of iterated wreath products we refer to \cite{MiasnikovVW17}). The crucial step there is the transfer result that the conjugacy problem in a wreath product $A \wr B$ is \TC-Turing-reducible to the conjugacy problems of $A$ and $B$ and the so-called power problem of $B$. 

	The \emph{power problem} of $G$ is defined as follows: on input  of $g,h \in G$ (as words over the generators) decide whether $h$ is a power of $g$ that is whether there is some $k\in \Z$ such that $g^k= h$ in $G$. In the ``yes'' case compute this $k$ in binary representation. If $g$ has finite order in $G$, the computed $k$ has to be the smallest non-negative such $k$. 

By \cite{MiasnikovVW17}, also the power problem of $A \wr B$ is \TC-Turing-reducible to the power problems of $A$ and $B$ given that torsion elements of $B$ have uniformly bounded order. The latter condition is also preserved by wreath products. 
Thus, in the light of \cite{MiasnikovVW17}, it remains to show that the power problem of nilpotent groups is in \TC and that the order of torsion elements is uniformly bounded, in order to establish the following theorem (note that \cite{MiasnikovVW17} is only for fixed groups; therefore, we formulate also the following results in a non-uniform setting):

\begin{theorem}
	Let $A$ and $B$ be finitely generated nilpotent groups and let $d \geq 1$, then the conjugacy problem of the $d$-fold iterated wreath products $A \wr^d B$ as well as $A \;{^d\wr}\; B$ is in \TC.
\end{theorem}

\begin{proof}
The following two lemmas together with a repeated application of Theorem 3, Lemma 5, and Theorem 5 of \cite{MiasnikovVW17}.
\end{proof}

\begin{lemma}
	Every finitely generated nilpotent group has a uniform bound on the order of torsion elements.
\end{lemma}
\begin{proof}
	We proceed by induction along a Mal'cev basis $(a_1, \dots, a_m)$ of $G$. If $a_1$ has infinite order, we are done by induction. Otherwise, let $k$ be the order of $a_1$ and $M$ be such that $g^M=1$ for all torsion elements $g \in \gen{a_2, \dots, a_m}$. Consider a torsion element $h \in \gen{a_1, \dots, a_m}$. Then $h^k \in \gen{a_2, \dots, a_m}$. Thus, $h^{kM} = 1$. Therefore, $kM$ is an upper bound on the order of torsion elements in $G$.
\end{proof}

\begin{lemma}
	For every finitely generated nilpotent group $G$, the power problem of $G$ is in uniform $\TC$. 
\end{lemma}
\begin{proof}
	We show a slightly more general statement by induction along a Mal'cev basis $(a_1, \dots, a_m)$ of $G$:  for every fixed arithmetic progression $\alpha + \beta \Z$, the power problem restricted to $\alpha + \beta \Z$ is in \TC, \ie given $g,h\in G$ it can be decided in \TC whether there is some $n\in \alpha + \beta \Z$ with $g^n=h$ in $G$ and, if so, that $n$ can be computed in \TC.

	We consider the input words $g$ and $h$ in the quotient $G/\oneset{a_2=\cdots=a_m=1}$. Let $g=a_1^k$ and $h=a_1^\ell$ in this quotient. If $k=\ell = 0$, it remains to solve the power problem in the subgroup $ \gen{a_2, \dots, a_m}$, which can be done by induction. Next, we distinguish the two cases that $a_1$ has infinite order and that it has finite order (in $G/\oneset{a_2=\cdots=a_m=1}$).
	
	In the case of infinite order, the only possible value for $n$ can be computed as $\ell / k$ (in \TC by \prettyref{thm:divisionTC}). If this is not an integer or not contained in the arithmetic progression (\ie\ $\ell / k \not\equiv \alpha \mod \beta$), then $h$ is not a power of $g$. Otherwise, one simply checks whether $g^{\ell / k}= h$ in $G$ (\ie\ solving the word problem). As $\ell$ is bounded by the input length by \prettyref{lem:malcev}, this can be done in \TC by \prettyref{thm:nilpotentExpWP}.
	
	In the case of finite order, let $d$ denote the order of $a_1$. It can be checked for all $0 \leq i < d$ in parallel whether $ki = \ell \mod d$. In case that there is such an $i$, the answer to the power problem is the same as the answer to the power problem in the subgroup $\gen{a_2, \dots, a_m}$ restricted to the arithmetic progression $ i + d\Z \cap \alpha +\beta\Z$ (the intersection can be hard-wired since there are only finitely many possibilities for a fixed group since the modulo is bounded by the least common multiple of the orders of finite order elements of the Mal'cev basis)~-- if there is no such $i$, the answer is ``no''.
\end{proof}

\section{Conclusion and Open Problem}

We have seen that most problems which in \cite{MacDonaldMNV15} were shown to be in \L indeed are in \TC even in the uniform setting where the number of generators and nilpotency class is fixed. Moreover, their binary versions are in $\TC(\ExtGCD)$ meaning that nilpotent groups are no more complicated than abelian groups in many algorithmic aspects. This contrasts with the slightly larger class of polycyclic groups: while the word problem is still in \TC \cite{Robinson93phd,KoenigL15}, the conjugacy problem is not even known to be in \NP. We conclude with some possible generalizations of our results:

\begin{question}
	Does a uniform version of \prettyref{thm:nilpotentExpWP} hold (\ie is the uniform word problem still in \TC) for fixed nilpotency class but an arbitrary number of generators? 
	
	What happens to the complexity if also the nilpotency class is part of the input? Note that in that case it is even not clear whether the word problem is still in polynomial time.
\end{question}

\begin{question}
	Is there a way to solve the conjugacy problem for nilpotent groups with binary exponents in \TC? Notice that we needed to compute greatest common divisors for solving the subgroup membership problem. However, there might be a way of solving the conjugacy problem using another method.
\end{question}

\begin{question}
What is the complexity of the uniform conjugacy problem when the nilpotency class is not fixed?
\end{question}

On the way for proving that the subgroup membership problem of nilpotent groups is in \TC, we established that the extended gcd problem with unary inputs and outputs is in \TC. However, the computed solution is not as small as the one computed by the \L algorithm from \cite{MH94}:

\begin{question}
	Is the following problem in \TC: given unary encoded numbers $a_1, \dots, a_n \in \Z$, compute $x_1 , \dots, x_n \in \Z$ with $\abs{x_i} \leq \frac{1}{2}\max\oneset{\abs{a_1}, \dots, \abs{a_n}}$ such that $x_{1} a_{1} + \dots + x_{n} a_{n} =  \gcd(a_{1},\ldots,a_{n})$?
\end{question}


\begin{thebibliography}{10}
	
	\bibitem{Blackburn}
	N.~Blackburn.
	\newblock Conjugacy in nilpotent groups.
	\newblock {\em Proceedings of the American Mathematical Society},
	16(1):143--148, 1965.
	
	\bibitem{boone59}
	W.~W. Boone.
	\newblock {T}he {W}ord {P}roblem.
	\newblock {\em Ann. of Math.}, 70(2):207--265, 1959.
	
	\bibitem{Dehn11}
	M.~Dehn.
	\newblock \"{U}ber unendliche diskontinuierliche {G}ruppen.
	\newblock {\em Math. Ann.}, 71(1):116--144, 1911.
	
	\bibitem{EickK04}
	B.~{Eick} and D.~{Kahrobaei}.
	\newblock {Polycyclic groups: A new platform for cryptology?}
	\newblock {\em ArXiv Mathematics e-prints}, 2004.
	
	\bibitem{ElberfeldJT11}
	M.~Elberfeld, A.~Jakoby, and T.~Tantau.
	\newblock Algorithmic meta theorems for circuit classes of constant and
	logarithmic depth.
	\newblock {\em Electronic Colloquium on Computational Complexity {(ECCC)}},
	18:128, 2011.
	
	\bibitem{GarretaMO16}
	A.~{Garreta}, A.~{Miasnikov}, and D.~{Ovchinnikov}.
	\newblock {Properties of random nilpotent groups}.
	\newblock {\em ArXiv e-prints}, Dec. 2016.
	
	\bibitem{Hal69}
	P.~Hall.
	\newblock {\em The {E}dmonton notes on nilpotent groups}.
	\newblock Queen Mary College Mathematics Notes. Mathematics Department, Queen
	Mary College, London, 1969.
	
	\bibitem{hesse01}
	W.~Hesse.
	\newblock Division is in uniform {TC}$^{0}$.
	\newblock In F.~Orejas, P.~G. Spirakis, and J.~van Leeuwen, editors, {\em
		ICALP}, volume 2076 of {\em Lecture Notes in Computer Science}, pages
	104--114. Springer, 2001.
	
	\bibitem{HeAlBa02}
	W.~Hesse, E.~Allender, and D.~A.~M. Barrington.
	\newblock Uniform constant-depth threshold circuits for division and iterated
	multiplication.
	\newblock {\em J.~Comput. Syst. Sci.}, 65:695--716, 2002.
	
	\bibitem{Kargapolov-Merzlyakov}
	M.~I. Kargapolov and J.~I. Merzljakov.
	\newblock {\em Fundamentals of the theory of groups}, volume~62 of {\em
		Graduate Texts in Mathematics}.
	\newblock Springer-Verlag, New York, 1979.
	\newblock Translated from the second Russian edition by Robert G. Burns.
	
	\bibitem{KRRRCh}
	M.~I. Kargapolov, V.~N. Remeslennikov, N.~S. Romanovskii, V.~A. Roman'kov, and
	V.~A. {\v{C}}urkin.
	\newblock Algorithmic questions for {$\sigma $}-powered groups.
	\newblock {\em Algebra i Logika}, 8:643--659, 1969.
	
	\bibitem{KoenigL15}
	D.~K\"onig and M.~Lohrey.
	\newblock Evaluating matrix circuits.
	\newblock In {\em Computing and combinatorics}, volume 9198 of {\em Lecture
		Notes in Comput. Sci.}, pages 235--248. Springer, Cham, 2015.
	
	\bibitem{LangeM98}
	K.~Lange and P.~McKenzie.
	\newblock On the complexity of free monoid morphisms.
	\newblock In K.~Chwa and O.~H. Ibarra, editors, {\em Algorithms and
		Computation, 9th International Symposium, {ISAAC} '98, Taejon, Korea,
		December 14-16, 1998, Proceedings}, volume 1533 of {\em Lecture Notes in
		Computer Science}, pages 247--256. Springer, 1998.
	
	\bibitem{LGS98}
	C.~R. {Leedham-Green} and L.~H. {Soicher}.
	\newblock Symbolic collection using {D}eep {T}hought.
	\newblock {\em LMS J. Comput. Math.}, 1:9--24 (electronic), 1998.
	
	\bibitem{Lipton_Zalc}
	R.~J. Lipton and Y.~Zalcstein.
	\newblock Word problems solvable in logspace.
	\newblock {\em J. ACM}, 24(3):522--526, July 1977.
	
	\bibitem{MacDonaldMNV15}
	J.~MacDonald, A.~G. Myasnikov, A.~Nikolaev, and S.~Vassileva.
	\newblock Logspace and compressed-word computations in nilpotent groups.
	\newblock {\em CoRR}, abs/1503.03888, 2015.
	
	\bibitem{MH94}
	B.~S. Majewski and G.~Havas.
	\newblock The complexity of greatest common divisor computations.
	\newblock In {\em Algorithmic number theory ({I}thaca, {NY}, 1994)}, volume 877
	of {\em Lecture Notes in Comput. Sci.}, pages 184--193. Springer, Berlin,
	1994.
	
	\bibitem{Malcev}
	A.~{Mal'cev}.
	\newblock {On homomorphisms onto finite groups}.
	\newblock {\em {Transl., Ser. 2, Am. Math. Soc.}}, 119:67--79, 1983.
	\newblock Translation from Uch. Zap. Ivanov. Gos. Pedagog Inst. 18, 49-60
	(1958).
	
	\bibitem{MiasnikovVW17}
	A.~Miasnikov, S.~Vassileva, and A.~Wei{\ss}.
	\newblock The conjugacy problem in free solvable groups and wreath product of
	abelian groups is in $\mathrm{TC}^{0}$.
	\newblock In P.~Weil, editor, {\em Computer Science - Theory and Applications -
		12th International Computer Science Symposium in Russia, {CSR} 2017, Kazan,
		Russia, June 8-12, 2017, Proceedings}, volume 10304 of {\em Lecture Notes in
		Computer Science}, pages 217--231. Springer, 2017.
	
	\bibitem{Mostowski}
	A.~Mostowski.
	\newblock Computational algorithms for deciding some problems for nilpotent
	groups.
	\newblock {\em Fundamenta Mathematicae}, 59(2):137--152, 1966.
	
	\bibitem{MNU2}
	A.~Myasnikov, A.~Nikolaev, and A.~Ushakov.
	\newblock The {P}ost correspondence problem in groups.
	\newblock {\em J. Group Theory}, 17(6):991--1008, 2014.
	
	\bibitem{MyasnikovNU16}
	A.~Myasnikov, A.~Nikolaev, and A.~Ushakov.
	\newblock Non-commutative lattice problems.
	\newblock {\em J. Group Theory}, 19(3):455--475, 2016.
	
	\bibitem{nov55}
	P.~S. Novikov.
	\newblock On the algorithmic unsolvability of the word problem in group theory.
	\newblock {\em Trudy Mat. Inst. Steklov}, pages 1--143, 1955.
	\newblock In Russian.
	
	\bibitem{Robinson93phd}
	D.~Robinson.
	\newblock {\em Parallel Algorithms for Group Word Problems}.
	\newblock PhD thesis, University of California, San Diego, 1993.
	
	\bibitem{Sim94}
	C.~C. Sims.
	\newblock {\em Computation with finitely presented groups}, volume~48 of {\em
		Encyclopedia of Mathematics and its Applications}.
	\newblock Cambridge University Press, Cambridge, 1994.
	
	\bibitem{Vollmer99}
	H.~Vollmer.
	\newblock {\em Introduction to Circuit Complexity}.
	\newblock Springer, Berlin, 1999.
	
\end{thebibliography}

\end{document}